 \documentclass[reqno]{amsart}
 \usepackage[latin1]{inputenc}
 \usepackage{amssymb}
 \usepackage{amsfonts}
 \usepackage{paralist}
 \usepackage{enumitem}
 \usepackage{indentfirst}
 \usepackage{amsmath}
 \usepackage{amsthm}
 \usepackage[all]{xy}
 \usepackage{mathrsfs}
 \usepackage{standalone}
 \usepackage{graphicx,psfrag,subfigure}
 \usepackage{pst-all}
 \usepackage{pifont}
 \usepackage[bookmarks]{hyperref}
 \usepackage{marvosym}
 \usepackage{pifont}
 \usepackage[normalem]{ulem}
\usepackage[bookmarks]{hyperref}
 \usepackage{cleveref} 
 \usepackage{wasysym}
\usepackage{accents} 
 \crefformat{section}{\S#2#1#3} 
 \crefformat{subsection}{\S#2#1#3}
 \crefformat{subsubsection}{\S#2#1#3}
 \crefrangeformat{section}{\S\S#3#1#4 to~#5#2#6}
 \crefmultiformat{section}{\S\S#2#1#3}{ and~#2#1#3}{, #2#1#3}{ and~#2#1#3}
 \def\K{\mathbb{K}}

 \def\R{\mathbb{R}}
 
 \def\T{\mathbb{T}}
 \def\dis{\displaystyle}

 \newcommand{\ubar}[1]{\underaccent{\bar}{#1}}
 \newtheorem*{thm-A}{Theorem A}
 \newtheorem*{thm-B}{Theorem B}
 \newtheorem*{thm-B'}{Theorem B'}
 \newtheorem*{thm-C}{Theorem C}
 \newtheorem*{thm-D}{Theorem D}
 \newtheorem*{cor*2}{Corollary 2}
 \newtheorem*{mainthm*}{Main Theorem} 
 \newtheorem*{lemma*}{Lemma}
 \newtheorem*{key lemma*}{Key lemma}
 \newtheorem*{keylemma*}{Key Lemma}
 \newtheorem{defi}{Definition}[section]
 \newtheorem{thm}{Theorem}[section]
 \newtheorem{prop}{Proposition}[section]
 \newtheorem{lemma}{Lemma}[section]
 \newtheorem{cor}{Corollary}[section]
 \newtheorem{rmk}{Remark}[section]
 
 \numberwithin{equation}{section}
 \newtheorem{theorem}{Theorem}
 \newtheorem{cor-thm}{Corollary}[theorem]
 
\begin{document}

 \title[Topological obstructions for robust transitivity]{Topological obstructions for robustly transitive endomorphisms on surfaces}

 \author[C. Lizana]{C. Lizana $\dagger$}
\address{$\dagger$ Departamento de Matem\'{a}tica. Instituto de Matem\'{a}tica
e Estat\'{i}stica.
Universidade Federal da Bahia. Av. Adhemar de Barros s/n, 40170-110. 
Salvador, Bahia, Brazil.} \email{clizana@ufba.br}


\author[W. Ranter]{W. Ranter $\ddagger$}
\address{$\ddagger$ Departamento de Matem\'{a}tica. Universidade Federal de Alagoas, Campus A.S. Simoes s/n, 57072-090. Macei\'o, Alagoas, Brazil.}
\address{$\ddagger$ Postdoctoral fellow at Math Section of ICTP. Strada Costiera, 11. I-34151, Trieste, Italy.}
\email{wagnerranter@gmail.com}


 \date{\today}

 \maketitle

\begin{abstract}
We address the problem of necessary conditions and topological obstructions for the existence of robustly transitive endomorphisms on surfaces. Concretely, we show that a weak form of hyperbolicity (namely, partial hyperbolicity) is a necessary condition in order to have robustly transitive displaying critical points, and  the only surfaces supporting this class of systems are either the torus or the Klein bottle. Furthermore, we also prove that the induced action by a partially hyperbolic endomorphism in the first homology group has at least one eigenvalue with modulus larger than one. 
\end{abstract}

\section{Introduction and Statement of the Main Results}\label{intro}

Throughout  this paper, unless specified, $M$ denotes a closed surface and $\mathrm{End}^1(M)$ the space of all the $C^1$-maps from $M$ into itself endowed with the usual $C^1$-topology. The elements of $\mathrm{End}^1(M)$ will be called \textit{endomorphisms}. Some of them display \textit{critical point}, that is, point such that the derivative is not an isomorphism; and the complement are endomorphisms without critical points which are local diffeomorphisms and diffeomorphisms. An endomorphism  $f \in \mathrm{End}^1(M)$ is said to be \textit{robustly transitive} if there exists a neighborhood $\mathcal{U}$ of $f$ in $\mathrm{End}^1(M)$ such that every $g \in \mathcal{U}$ is transitive, recalling that \textit{transitive} means that such map has a dense forward orbit in the whole surface.

The main aim of our paper is to give necessary conditions and some topological obstructions for the existence of robustly transitive surface endomorphisms displaying critical points.
The first result we present is in regard of necessary conditions, we show that in the case of surface endomorphisms displaying critical points, a weak form of
hyperbolicity is necessary for robust transitivity. Concretely,

\begin{theorem}\label{thm-A}
Every robustly transitive surface endomorphism displaying critical points is a partially hyperbolic endomorphism.
\end{theorem}

The definition of partial hyperbolicity for endomorphisms requires some preliminary notions about cone-field.
Recall that a \textit{cone-field} $\mathscr{C}$ on $M$ is a family of closed convex non-vanishing cone $\mathscr{C}(x) \subseteq T_xM$ at each point $x$ in $M$. A cone-field $\mathscr{C}$ is said to be \textit{invariant} (or \textit{k-invariant}, if we want to emphasize the role of $k$) if $Df_x^k(\mathscr{C}(x))$ is contained in the interior of $\mathscr{C}(f^k(x))$ for all $x$ in $M$ which will be denoted by $\mathrm{int}(\mathscr{C}(f^k(x)))$. Finally, an endomorphism $f$ is a \textit{partially hyperbolic endomorphism} if there exist $\lambda >1, \ell \geq 1,$ and an invariant cone-field $\mathscr{C}$ satisfying:
\begin{itemize}
\item[[PH1\!\!]] \textit{Transversal to the kernel:} for all $x\in M$ and $n\geq 1$,
$$\ker(Df_x^n)\cap \mathscr{C}(x)=\{0\};$$
\item[[PH2\!\!]] \textit{Unstable property:} for all $x\in M$ and $v \in \mathscr{C}(x)$, $$ \|Df^{\ell}_x(v)\|\geq \lambda\|v\|.$$ 
\end{itemize}

Let us call this cone-field by \textit{unstable cone-field} for $f$ and denote it by $\mathscr{C}^u$. It should be noted that, up to an iterated, we can assume that the unstable cone-field $\mathscr{C}^u$ is also $\ell$-invariant. If an invariant cone-field satisfies only the transversality property ([PH1]), then we say that $f$ \textit{admits a dominated splitting}. 


\medskip
 Let us briefly comment on the present state of the art. 
 Ma\~{n}\'{e} proved  
 in \cite{Mane-Closinglemma}   that every robustly transitive surface diffeomorphism admits a hyperbolic structure in the whole surface. Ma\~{n}\'{e}'s result has no direct generalization to higher dimension. D\'{i}az-Pujals-Ures and Bonatti-D\'{i}az-Pujals proved in \cite{DPU} and \cite{BDP},  for manifolds of dimension three and larger than four, respectively, proved that robust transitivity implies a weak form of hyperbolicity, called 
 \textit{volume hyperbolicity}\footnote{A diffeomorphism is volume hyperbolic if the invariant splitting $TM=E_1\oplus \cdots \oplus E_k$ is dominated and $Df$ restricts to $E_1$ and $E_k$ are volume contracting and volume expanding, respectively. In particular, it implies partial hyperbolicity in three-dimension.}. For local diffeomorphisms, Lizana-Pujals gave in \cite{LP}  necessary and sufficient conditions for the existence of robustly transitive, and, in particular, it is shown that no form of weak hyperbolicity is necessary for robust transitivity. However, for endomorphisms admitting critical points, the robust transitivity requires, as for diffeomorphisms, a weak form of hyperbolicity as it is stated in Theorem \ref{thm-A} above.

The existence of an invariant cone-field on a surface provides some topological obstructions. The next result gives a classification answering  a natural question, which surfaces support robustly transitive endomorphisms.

\begin{theorem} \label{thm-B}
The only surfaces that might admit  robustly transitive endomorphisms are either the torus $\T^2$ or the Klein bottle $\K^2$.
\end{theorem}


Some comments are in order. For diffeomorphisms, it follows from \cite{Mane-Closinglemma} that the only surface admitting a robustly transitive is the 2-torus.
For local diffeomorphisms, it is well-known that 
these maps are covering maps and the only surfaces that can cover themselves are either the torus $\T^2$ or the Klein bottle $\K^2$. However, it was not known, until now,  if the existence of robustly transitive endomorphisms displaying critical points implies topological obstructions. 
Although some examples of robustly transitive local diffeomorphisms in the torus $\T^2$ and the Klein bottle $\K^2$ (as expanding endomorphisms \footnote{$f \in \mathrm{End}^1(M)$ is an expanding endomorphisms if $\|Df_x(v)\|> \|v\|$ for all $x \in M$ and $v \in T_xM$.} for instance) are well-known, examples of robustly transitive endomorphisms admitting critical points first appear in \cite{Berger-Rovella} and \cite{ILP} on the torus $\mathbb{T}^2$; and recently in \cite{CW2}, we built a new class of robustly transitive endomorphisms exhibiting persistent critical points on the 2-torus 
and on the Klein bottle, based on the  geometric construction developed by Bonatti-D\'{i}az 
 in \cite{BD} to produce robustly transitive diffeomorphisms.  The proof of Theorem \ref{thm-B} for endomorphisms displaying critical points use a classical argument provided the existence of continuous subbundles over $M$, which follows from Theorem \ref{thm-A}, and it will be presented in Section \cref{sec:main-thm-consequences}.

An immediate and interesting consequence from Theorem \ref{thm-B} is the following.

\begin{cor-thm}
There are not robustly transitive on the sphere $\mathbb{S}^2$.
\end{cor-thm}

The third result we present now is in regard of topological obstructions as well.
More precisely,

\begin{theorem}\label{thm-C}
The action of a transitive endomorphism admitting a dominated splitting in the first homology group of $M$ has at least one eigenvalue with modulus larger than one.
\end{theorem}

In particular, the proof of Theorem \ref{thm-C} will allow us to conclude the following.

\begin{cor-thm}\label{cor1}
The action of partially hyperbolic endomorphisms in the first homology group of $M$ has at least one eigenvalue with modulus larger than one.
\end{cor-thm}

From Theorem \ref{thm-C} we get the following corollary.

\begin{cor-thm}\label{cor2}
The action of robustly transitive endomorphisms
in the first homology group of $M$ has at least one eigenvalue with modulus larger than one.
\end{cor-thm}

Consequently, one obtains the following.

\begin{cor-thm}\label{cor3}
There are not robustly transitive surface endomorphisms homotopic to the identity.
\end{cor-thm}


Let us comment the main ideas of the proof of Theorem \ref{thm-C}. First we prove that the length of the iterated of any tangent arcs  to the cone-field grows exponentially; and later, based on \cite{BBI04} where 
it is proved that the action of partially hyperbolic diffeomorphisms on three manifolds on the first homology group is partially hyperbolic as well, we prove that if the action has some  
 eigenvalue of modulus greater than one, then there exists an arc such that the length of its iterated grows sub-exponentially. In order to prove the first part, we  adapt the arguments of \cite{Pujals-Sambarino}  to our setting. 
For further details and the proof of Corollaries \ref{cor1} and \ref{cor2}  see Section~\cref{homotopy}.

\medskip


Finally, we introduce the main result of this work giving a dichotomy that will be used to prove Theorem \ref{thm-A}. Before continuing, let us fix the following notation and introduce some previous results.  We denote the set of all the critical points of $f$ by $\mathrm{Cr}(f)$ and by $\mathrm{int}(\mathrm{Cr}(f))$ its interior; further, $|\ker(Df)|$ denotes the maximum of $|\ker(Df_x)|$ over $M$, where $|\ker(Df_x)|$ is the dimension of $\ker(Df_x)$. An endomorphism 
$f \in \mathrm{End}^1(M)$ is said to have \textit{full-dimensional kernel} if there exist an integer $n \geq 1$ and $x \in \mathrm{Cr}(f)$ such that 
\begin{align}
|\ker(Df^n_x)|=\dim M.
\end{align}
This property is an obstruction for robust transitivity and will play an important role in our approach. More precisely,  

\begin{keylemma*}[Full-dimensional kernel obstruction]\label{key obstruction}
Let $f: M \to M$ be an endomorphism having full-dimensional kernel. Then $f$ cannot be a robustly transitive endomorphism. 
\end{keylemma*}


Since the proof is not difficult we present it as follows.

%
%

\begin{proof}[Proof of Key Lemma]
Since $|\ker(Df_x^n)|=2$, one has that $Df_x^n=0$. Then, by \hyperref[Franks-lemma]{Franks' Lemma}, there exists $\tilde{f}$ $C^1$-close to $f$ such that $\tilde{f}^n$ behave as $Df_x^n$ in a neighborhood of $x$, and so $\tilde{f}^n$ is constant in a neighborhood of $x$ implying that $\tilde{f}$ cannot be transitive.
\end{proof}

An interesting fact, that will be shown in Section \cref{sec:main-thm}, is the following dichotomy 
which appears ``naturally" as an obstruction for domination property. 

\begin{mainthm*}\label{main-theo}
Let $f: M \to M$ be a transitive endomorphism whose $\mathrm{int}(\mathrm{Cr}(f))$ is nonempty and $|\ker (Df^n)|=1$ for every $n\geq 1$. Then,
\begin{itemize}
\item either $f$ admits a dominated splitting; 
\item or $f$ can be approximated by endomorphisms having full-dimensional kernel.
\end{itemize}
\end{mainthm*}



Some comments related to the \hyperref[main-theo]{Main Theorem} are in order. Some similarities with the diffeomorphisms setting can be observed at this result. For diffeomorphisms, the dichotomy guarantees that either one has domination property or one can create, up to a perturbation, some obstructions (sinks and sources) for robust transitivity. In both setting, one starts proving the existence of a dominated splitting over a properly subset which can be extended to the whole manifold. The choices of obstructions for robust transitivity and the properly subset play a fundamental role in the proof; the choices are related in a way that allows us to define an invariant splitting over a properly subset which the absence of domination property (see
Definition~\ref{def-DS}) implies an obstruction for robust transitivity. Let us make a brief discussion about previous results.

\medskip

$\bullet$ For diffeomorphisms, the  used obstructions for robustly transitive  are the existence of \textit{sinks} or \textit{sources}\footnote{A sink (source) is a periodic point (e.g., $f^n(p)=p$) whose  linear map $L=Df^n_p$ has all the eigenvalues with modulus less (greater) than one. The existence of sinks (sources) implies the existence of a neighborhood $U$ of $p$ such that $f^n(\overline{U})\subseteq U \, (f^{-n}(\overline{U})\subseteq U$) which contradicts the transitivity.}; and the properly subset is the set of all periodic points in \cite{Mane-Closinglemma}, and it is the homoclinic class of a saddle point in \cite{DPU,BDP}.

\medskip

\begin{itemize}
\item[-] In \cite{Mane-Closinglemma}, 
the obstructions 
mentioned above implies that the periodic points of a robustly transitive diffeomorphism are hyperbolic. Then, one can define a ``natural" invariant splitting over the periodic points. 
It is shown that 
under the absence of domination property, sinks and sources can be created, 
 up to a perturbation, which is a contradiction. Therefore, the hyperbolic splitting over the periodic points is dominated and can be extended to the closure of the periodic points that generically is the whole surface.

\medskip

\item[-] In \cite{DPU} and \cite{BDP}, the properly subset is the homoclinic class of a hyperbolic saddle point. It has a ``natural" splitting given by the transversal intersection of the stable and unstable manifolds of such saddle. They prove that the absence of domination of such splitting implies  the existence of sinks and sources for a perturbation which is an obstruction for robust transitivity. Finally, they use classical results, such as Closing Lemma and Connecting Lemma, to extend the splitting to the whole manifold. 
\end{itemize}


$\bullet$ For non-invertible endomorphisms, observe that just sinks keep being obstructions for robustly transitive (e.g., expanding endomorphisms are robustly transitive and have sources). However, in the setting of the \hyperref[main-theo]{Main Theorem}, the alternative obstruction for robustly transitive, full-dimensional kernel allows us to define a properly subset as the set of all full orbits (see Section~\cref{section-ph}) which enter infinitely many times for the past/future in the set of the critical points. Concretely, 
\begin{align}\label{Lambda-set}\tag{$\ast$}
\Lambda_f=\left\{(x_i)_i \subseteq M_f \left| \begin{array}{ll}
 f(x_i)=x_{i+1}; \ \ \text{and} \ \ x_i \in \mathrm{Cr}(f)\ \
\text{for}\\ \text{infinitely many negative/positive} \ \ i \in \mathbb{Z}
\end{array} \right.\right\}.
\end{align}

In general, the authors do not know if this set is ``typically" nonempty. However, it will be shown in Section \cref{sec:main-thm} that if $f$ satisfies the hypothesis of \hyperref[main-theo]{Main Theorem} then $\Lambda_f$ is a dense subset, further details in Section \cref{construction-lambda}. More generally, in Section \cref{sec:main-thm-consequences} (see Lemma \ref{typically}), it will be shown that every robustly transitive endomorphism $f$ displaying critical points can be approximated by endomorphisms $g$
whose corresponding $\Lambda_g$ is dense,  where $\Lambda_g$ refers to the set  associated to $g$ given as in \eqref{Lambda-set}.
Hence, we first assume that $\Lambda_f$ is a nonempty set and define the following subbundles, 
\begin{align}\label{def:EF}\tag{$\ast \ast$}
E(x_j)=\ker\Bigl(Df^{^{\tau^{+}_j+1}}_{_{x_j}}\Bigl) \ \ \text{and} \ \ F(x_j)=\mathrm{Im}\Bigr(Df^{^{|\tau^{-}_j|}}_{_{x_{\tiny{\tau_j^{-}}}}}\Bigl),
\end{align}
where $\tau^{+}_j$ is the first time the orbit $(x_{j+i})_i$ enters in the set of critical points and $\tau^{-}_j$ is the first time before the orbit $(x_{j+i})_i$ leaves the set of critical points. In Section \cref{sec:main-thm}, we will prove that either $E\oplus F$ over $\Lambda_f$ is dominated in the sense of Definition \ref{def-DS}; or $f$ can be approximated by an endomorphism having full-dimensional kernel. Finally, we use the equivalence between Definition \ref{def-DS} and the definition of dominated splitting given
at the beginning of this section ([PH1]), see Section \cref{section-ph} for details, to conclude the proof of \hyperref[main-theo]{Main Theorem}. 

The novelty in our approach are the full-dimensional kernel obstruction for robustly transitive, the choice of $\Lambda_f$ and the construction of the splitting $E\oplus F$ over $\Lambda_f$. Furthermore, it will not be used any classical result such as Closing Lemma and Connecting Lemma. These are the most important difference between our approach and the approaches in \cite{Mane-Closinglemma,DPU} and \cite{BDP}. 

Finally, Section~\cref{sec:exp direction} is devoted to prove that the dominated splitting obtained in Section~\cref{sec:main-thm-consequences} is in fact partially hyperbolic, that is, the extremal dominating bundle is expanding, finishing the proof of Theorem \ref{thm-A}. 

\subsection{Sketch of the proof} The strategy for proving the results stated above is as follows.  Consider a robustly transitive endomorphism displaying critical points.
\begin{enumerate}
\item The existence of critical points and transitivity allows us to construct a ``properly subset" $\Lambda_f$ defined by \eqref{Lambda-set} (Section \cref{construction-lambda}). 

\item Observe that the kernel of the differential is at most one dimensional. If not, we are able to find a perturbation having an attracting periodic point, which contradicts the transitivity of nearby maps (Key Lemma).

\item Previous observation provide us a ``natural" candidate  for an invariant splitting over the critical points, since one of the subbundle is determined by the kernel of the differential over the critical points and the other subbundle is given by a ``transversal'' direction to the kernel. Then, this splitting is extended to the iterates and can be defined over the ``properly subset" as in \eqref{EF-splitting} (Section \cref{construction-lambda} and \cref{pf-key-thm}). 

\item For every nearby map to the initial holds that the angle between the invariant subbundles over the ``properly set" is uniform bounded away from zero. 
Our approach for proving it is a little bit different from the diffeomorphism setting approach in   \cite{Mane-Closinglemma,DPU, BDP}. For invertible maps, if the angle goes to zero, sinks and sources can be created, and both are obstructions for transitivity. However, in our setting sources are not an obstruction for transitivity, instead we use the fact of having full-dimensional kernel, contradicting that the initial map is robustly transitive (Section \cref{pf-key-thm}). 

\item Then,  the splitting is, in fact, a dominated splitting over the ``properly set". Furthermore, this dominated splitting is extended to the whole surface and it is an open property. The strategy for proving the domination property is to prove, first, that the domination property holds in a neighborhood of the critical points, second, it is extended to the ``properly set'', and then it holds for the whole surface. 
For this, roughly speaking, assuming that the domination property does not hold, then we can find a perturbation having full-dimensional kernel, contradicting the fact that the initial map is robustly transitive (Section \cref{uniform-DP} and \cref{sec:main-thm-consequences}).


\item Once we have  the existence of the dominated splitting, we get topological obstructions providing which surfaces support robustly transitive endomorphisms (Section \cref{sec:main-thm-consequences}).

\item Next, we prove that this splitting has a topological expanding behavior in the direction transversal to the kernel. Indeed, we show 
that the length of the iterates of arcs  tangent to the extremal dominate bundle grows exponentially
(Section \cref{homotopy}).

\item Finally, we show that the extremal dominate bundle is, in fact, an uniform expanding bundle getting that the initial map is partially hyperbolic as we wished (Section \cref{sec:exp direction}). For this,
 we assume by contradiction that the extremal dominate subbundle is not expanding. Then, the domination property implies that the dominated subbundle is contracting, and, up to a perturbation, there is a sink, contradicting that the initial map is  robustly transitive. 
\end{enumerate}

\section{Preliminaries}\label{section-ph}

This section is devoted to introduce the basic notions that will be used along this paper, such as the notion of partial hyperbolicity for endomorphisms displaying critical points in term of invariant splitting. Moreover, it will also be presented some fundamental properties for dominated splitting as uniqueness, continuity of the splitting, and the equivalence between the both notions of partial hyperbolicity for endomorphisms in terms of cone-field and invariant splitting. For the sake of clearness,  we prove that both definition are equivalent and then use indistinctly any of them according to the situation, it is important to recall that since we are in a non-invertible setting, we are dealing with several pre-images and has to be careful when we consider backward iterates.  The readers that are familiar with these notions can skip this section, and come back in case it is necessary.


\subsection{Dominated splitting and partial hyperbolicity for endomorphisms}\label{subsection-DS}

Due to the fact that an endomorphism might have several pre-images, it is appropriated to define a splitting over the space of the full orbits (or simply, orbit). Let us denote by $M^{\mathbb{Z}}$ the product space and, for every $f \in \mathrm{End}^1(M)$, define the space of all the orbits of $f$ by
\begin{align}
M_f=\{(x_i)_i \in M^{\mathbb{Z}}: f(x_i)=x_{i+1}, \forall i \in \mathbb{Z}\}.
\end{align}

Denote by $TM_f$ the vector bundle over $M_f$ defined as the pullback of $TM$ by the projection $\pi_0:M_f \to M$ defined by  $\pi_0((x_i)_i)=x_0$. In particular, it is well known that $TM_f$ and $TM$ are isomorphic. Unless specified, we use $T_{x}M$ to denote the fiber of $TM$ at $x \in M$ and the fiber of $TM_f$ at $(x_i)_i \in M_f$ where $x_0=x$. Furthermore, it should be noted that $f$ acts in $M_f$ as a shift map $(x_i)_i\mapsto (x_{i+1})_i$ and the derivative acts from $T_{x_i}M$ to $T_{x_{i+1}}M$. Then, we say that $\Lambda$ a subset  of $M_f$ is \textit{f-invariant} if it is invariant for the shift map. Before defining dominated splitting (in term of splitting) for endomorphisms displaying critical points, we recall the definition for local diffeomorphisms (or diffeomorphism).

\medskip

A $f$-invariant subset $\Lambda$ in $M_f$ is said to admit a dominated splitting for $f$ if there exist $\alpha>0,\, \ell \geq 1,$ and two one-dimensional bundles $E$ and $F$  such that for all $(x_i)_i \in \Lambda$ and $i \in \mathbb{Z}$ hold

\begin{itemize}
\item[-] \textit{Invariance splitting:}
\begin{align*}
& Df(E(x_i))= E(f(x_i)) \ \ \text{and} \ \ Df(F(x_i))=F(f(x_i));\\ & \text{and}, \ \  T_{x_i}M=E(x_i)\oplus F(x_i); 
\end{align*}

\smallskip

\item[-] \textit{Uniform angle:} $\varangle(E(x_i),F(x_i))\geq \alpha$;

\smallskip

\item[-] \textit{Domination Property:}
\begin{align*}
\|Df^{\ell} \mid_{E(x_i)}\|\leq \frac{1}{2}\|Df^{\ell}\mid_{F(x_i)}\|.
\end{align*}
\end{itemize}
$Df^{\ell}\mid_{E(x)}$ and $Df^{\ell}\mid_{F(x)}$ denote the restriction of $Df^{\ell}_x$ to the linear subspaces $E(x)$ and $F(x)$, respectively.

\begin{rmk}\label{rmk:angles}
\begin{itemize}
\item[]
\item[-] In this setting, it is known that as $T_xM=E(x)\oplus F(x)$, there exists a linear map $\phi_x:E(x) \to E^{\perp}(x)$ such that the graph of $\phi_x$ is the subspace $F(x)$. Then, we define the angle between $E(x)$ and $F(x)$ by
$$\|\phi_x\|=\tan \varangle(E(x),F(x)).$$

\smallskip

\item[-] Although we denote the subbundles by $E(x)$ and $F(x)$, we would like to emphasize that in general $E$ depends on the forward orbits $($i.e., $(x_i)_{i\geq 0}$ whose $x_0=x$$)$ and $F$ depends on the backward orbits $($i.e., $(x_i)_{i\leq 0}$ whose $x_0=x$$)$. 
\end{itemize}
\end{rmk}

However, when critical points appear, it is natural to think that the invariance property is affected. Indeed, let $T_xM=E(x)\oplus F(x)$ be satisfying the properties above. It should be noted that if $x \in \mathrm{Cr}(f)$ we must have that $E(x)$ is the kernel of $Df$ at $x$. Otherwise, either $Df_x(F(x))=\{0\}$ or $E(f(x))=F(f(x)),$ contradicting the definition above. Thus, a natural extension of the dominated splitting definition for endomorphisms displaying critical points is as follows.

\begin{defi}[Dominated splitting]\label{def-DS}
Let $\Lambda$ be a $f$-invariant subset of $M_f$. We say that $\Lambda$ admits a dominated splitting for $f$ if there exist $\alpha>0,\, \ell \geq 1,$ and two one-dimensional bundles $E$ and $F$  such that for every $(x_i)_i \in \Lambda$ and $i \in \mathbb{Z}$ hold
\begin{itemize}
\item[-] \textbf{Invariance splitting:} 
\begin{align*}
& Df(E(x_i))\subseteq  E(f(x_i)) \ \ \text{and} \ \ Df(F(x_i))=F(f(x_i));\\ & \text{and} \ \  T_xM=E(x_i)\oplus F(x_i);
\end{align*}

\smallskip

\item[-] \textbf{Uniform angle:} $\varangle(E(x_i),F(x_i))\geq \alpha$;

\smallskip

\item[-] \textbf{Domination Property:}
\begin{align}\label{eq-DS-property}
\|Df^{\ell} \mid_{E(x_i)}\|\leq \frac{1}{2}\|Df^{\ell}\mid_{F(x_i)}\|.
\end{align}
We denote the domination property by $E \prec_{\ell} F$.
\end{itemize}
\end{defi}

If $\Lambda=M_f$, we say $M_f$ admits a dominated splitting or simply $f$ has a dominated splitting. Unless specified, we use $T_{\Lambda}M=E\oplus F$ to denote the splitting over $\Lambda$ and for each point $(x_i)_i \in \Lambda$ the splitting $E\oplus F$ over the orbit $(x_i)_i$ means that
$$\bigsqcup_{i=-\infty}^{\infty} T_{x_i}M= \bigsqcup_{i=-\infty}^{\infty} E(x_i)\oplus F(x_i).$$

\begin{defi}[Partially hyperbolic]
We say that an endomorphism $f$ displaying critical points is partially hyperbolic if there exist two 
one-dimensional bundles $E$ and $F$ satisfying: 

\medskip

\begin{itemize}
\item[-]\textbf{Dominated splitting:} $TM_f=E\oplus F$ is a dominated splitting for $f$;

\medskip

\item[-]\textbf{Unstable direction:}  there exist $k \geq 1$ and $\lambda>1$ such that for all $(x_i) \in M_f$ and $i \in \mathbb{Z}$ hold that
\begin{align}\label{ph-eq}
\|Df^{k}\mid_{F(x_i)}\|\geq \lambda.
\end{align}
\end{itemize}
\end{defi}
It should be noted that, up to an iterated, we can assume that $k\geq 1$ in \eqref{ph-eq} and $\ell\geq 1$ in \eqref{eq-DS-property} are the same.

\subsection{Fundamental properties of dominated splitting}

Here we present some fundamental properties of dominated splitting.
The uniqueness of dominated splitting holds for endomorphisms without critical points, for details see \cite{Crovisier-Potrie}. Next proposition guarantees, even in the setting with critical points, the uniqueness of the splitting.

\begin{prop}[Uniqueness]\label{uniqueness} 
The dominated splitting $T_{\Lambda} M=E\oplus F$ is unique. That is, if $\Lambda$ admits two dominated splittings $E\oplus F$ and $G\oplus H$ for $f$, we must have $E(x_i)=G(x_i)$ and $F(x_i)=H(x_i)$ for all $(x_i)_i \in \Lambda$ and $i \in \mathbb{Z}$.
\end{prop}


\begin{proof}
For every $(x_i)_i$ in $\Lambda$, we can consider two cases. 

\medskip

\textit{Case I:}  the orbit $(x_i)_i$ such that $x_i \notin \mathrm{Cr}(f)$. 

\smallskip

In this situation, the proof follows from same arguments as in the invertible case and can be found in \cite{Crovisier-Potrie}.

\smallskip

\textit{Case II:} the orbit $(x_i)_i$ such that $x_i \in \mathrm{Cr}(f)$ for some $i \in \mathbb{Z}$. 

\smallskip

Here, without loss of generality, we can assume that $x_0 \in \mathrm{Cr}(f)$ and $x_i \notin \mathrm{Cr}(f)$ for every $i\neq 0$. Thus, one has that $E(x_0)=\ker(Df_{x_0})=G(x_0)$ and $Df_{x_i}$ is an isomorphism for $i\neq 0$ which imply $E(x_i)=G(x_i)$ for $i\leq 0$ and $F(x_i)=H(x_i)$ for $i \geq 1$. In particular, it should be noted that $T_{x_i}M$ admits two dominated splitting $E(x_i)\oplus F(x_i)$ and $E(x_i)\oplus H(x_i)$ for an invertible map $Df$ on $(x_i)_{i\leq 0}$ and $T_{x_i}M$ admits two dominated splitting $E(x_i)\oplus F(x_i)$ and $G(x_i)\oplus F(x_i)$ for the isomorphism $Df$ on $(x_i)_{i\geq 1}$. Therefore, repeating the same argument for the invertible case one can conclude that $E(x_i)=G(x_i)$ and $F(x_i)=H(x_i)$ for every $i \in \mathbb{Z}$.
\end{proof}

The following result shows the continuity of the dominated splitting. Moreover, it proves that a dominated splitting can be extended to the closure.

\begin{prop}[Continuity and extension to the closure]\label{continuity&extension}
The map 
\begin{align}\label{EF-map}
\Lambda \ni (x_i)_i \longmapsto \,  E(x_0)\oplus F(x_0)
\end{align}
is continuous. Moreover, it can be extended to the closure of $\Lambda$ continuously.
\end{prop}

\begin{proof}
Let $(x_i^n)_i \to (x_i)_i$ as $n \to \infty$ with $(x_i^n)_i$ and $(x_i)_i$ in $\Lambda$. Let $v_n$ and $w_n$ be unit vectors in $E(x_0^n)$ and $F(x_0^n)$, respectively. Up to a subsequence, we can assume that $v_i^n$ and $w_i^n$ converge to $v_i$ and $w_i$ such that the spanned spaces are $\widetilde{E}(x_i)$ and $\widetilde{F}(x_i)$, respectively. In particular, using that $f$ is a $C^1$-map and the angle between $E$ and $F$ is uniformly bounded away from zero, we obtain that $\widetilde{E}\oplus \widetilde{F}$ is a dominated splitting over $(x_i)_i$. Thus, by uniqueness of the existence of dominated splitting, we get that 
$\widetilde{E}(x_i)=E(x_i)$ and $\widetilde{F}(x_i)=F(x_i)$ which implies that the argument does not depend on the choice of the subsequence of $(v_i^n)_n$ and $(w_i^n)_n$. Therefore, the map in \eqref{EF-map} is continuous.
By the continuity of the map in \eqref{EF-map} and uniqueness of the dominated splitting, we can extend $E\oplus F$ to the closure of $\Lambda$ as the limits of $E$ and $F$.
\end{proof}


\begin{rmk}\label{proj-E}
It is well-known and follows from the proof of uniqueness that if  $E\oplus F$ is a dominated splitting, then the subbundle $E$ only depends on the forward orbits. That is, for every $(x_i)_i$ and $(y_i)_i$ in $\Lambda$, one has for every $i \geq 0$ that
 $$E(x_i)=E(y_i), \ \ \text{whenever} \ \ x_0=y_0.$$
In particular, by Proposition \ref{continuity&extension}, the subbundle induced by $E$ on $M$ is continuous which, by slight abuse of notation, we also denote by $E$.
\end{rmk}


The following is a characterization of partial hyperbolicity which will be useful for proving Theorem \ref{thm-A}. However, the proof is a standard argument.
\begin{prop} \label{ph-end}
Let $TM_f=E\oplus F$ be a dominated splitting for $f$. There exist $\ell_0\geq 1$ and $\lambda>1$ such that for every $(x_i)_i \in M_f$ with $x_0=x$, there is $1\leq k \leq {\ell}_0$ with $k$ depending of $(x_i)_i$ so that
\begin{align}
\|Df^k\mid_{F(x)}\|\geq \lambda
\end{align}
if, and only if, $f$ is a partially hyperbolic endomorphism.
\end{prop}

\begin{proof} Since $M_f \ni (x_i)_i \mapsto F(x_0)$ is continuous, we can take $C=\min\{\|Df^j\mid_{F(x_0)}\|:1\leq j \leq \ell_0-1 \,\, \text{and} \,\, (x_i)_i \in M_f\}$ 
and write $\ell=k\ell_0+r,$ $0 \leq r \leq \ell_0-1$. Then we have  for $ r=1,\dots,\ell_0-1,$
\begin{align*}
\|Df^{\ell}\mid_{F(x_0)}\|\geq C\lambda^k.
\end{align*}
Hence, taking $k_0\geq 1$ such that $\lambda_0:=C\lambda^{k_0}>1$, we have that $\|Df^{\ell}\mid_{F(x_0)}\|\geq \lambda_0$. The reciprocal is clear.
\end{proof}



\subsection{Cone-criterion} 
Here, we present the equivalence between the definitions of dominated splitting and partial hyperbolicity in terms of cone-field and  invariant splitting. 

Let $E$ be a $Df$-invariant continuous subbundle of $TM$. For every $x \in M$, we define the following cone-field $\mathscr{C}_{E}$ on $M$ with core $E$ of length (angle) $\eta >0$,
\begin{align}
\mathscr{C}_{E}(x,\eta)=\{u_1+u_2 \in E(x)\oplus E(x)^{\perp}:\|u_2\|\leq \eta \|u_1\|\}.
\end{align}
By Remark \ref{rmk:angles} it can be rewritten as:
$$\mathscr{C}_{E}(x,\eta)=\{v \in E(x)\oplus E(x)^{\perp}:\varangle(E(x),\mathbb{R}\langle v\rangle)\leq \eta\},$$
where $\mathbb{R}\langle v\rangle$ denotes the subspace generated by $v$. The \textit{dual cone} of $\mathscr{C}_{E}(x,\eta)$ is the cone-field given by
\begin{align}
\mathscr{C}_{E}^{\ast}(x,\eta)=\{u_1+u_2 \in E(x)\oplus E(x)^{\perp}:\|u_1\|\leq \eta^{-1} \|u_2\|\}.
\end{align}
In other words, $\mathscr{C}_{E}^{\ast}(x,\eta)$ is the closure of $T_xM\backslash \mathscr{C}_{E}(x,\eta)$. The interior of $\mathscr{C}_{E}$ is given by $$\mathrm{int}(\mathscr{C}_{E}(x,\eta))=\{u_1+u_2 \in E(x)\oplus E(x)^{\perp}:\|u_1\| < \eta \|u_2\|\}\cup \{0\}.$$

By Remark \ref{proj-E}, the subbundle induced by $E$ of a dominated splitting $E\oplus F$ for $f$ over $M_f$ is  well-defined. Then, consider the cone-field $\mathscr{C}_E$ on $M$ and state the following equivalence between the definitions of dominated splitting in terms of cone-fields and invariant splitting.

\begin{prop}\label{cone-criterion}
$E\oplus F$ is a dominated splitting for $f$ if, and only if,
the cone-field $\mathscr{C}_E$ on $M$ is invariant and transversal to the kernel.
\end{prop}

\begin{proof}
Since the angle between $E$ and $F$ is bounded away from zero, there is a number $\alpha >0$ verifying that for every $x \in M$, the direction $F$ is not contained in $\mathscr{C}_E(x,\alpha)$. In particular, we have that for every $(x_i)_i\in M_f, \, F(x_i) \subseteq \mathscr{C}_E^{\ast}(x_i,\alpha)$. The domination property implies that, for $k\geq 1$ large enough, the direction $Df^k(E^{\perp}(x_i))$ gets closer to $F(f^k(x_i))$ and $\|Df^k\mid_{E^{\perp}(x_i)}\|\approx \|Df^k\mid_{F(x_i)}\|$. Hence, one can fix $k\geq 1$ so that $Df^k(\mathscr{C}_E^{\ast}(x_i,\alpha)) \subseteq \mathrm{int}(\mathscr{C}_E^{\ast}(f^k(x_i),\alpha))$. This proves the necessary condition. The proof for the sufficient condition can be found in \cite[Section 2]{Crovisier-Potrie}.
\end{proof}

\begin{rmk}
It should be emphasized that the existence of a dominated splitting is an open property in the $C^1$ topology. That is, if $f$ admits an invariant cone-field $\mathscr{C}$ transversal to the kernel, then there exists a neighborhood $\mathcal{U}$ of $f$ in $\mathrm{End}^1(M)$ such that the cone-field $\mathscr{C}$ is invariant and transversal to the kernel for each $g \in \mathcal{U}$ as well.
\end{rmk}

The following result shows the equivalence between the definitions of partially hyperbolic endomorphisms in terms of cone-field and invariant splitting. The proof can be found in \cite{Crovisier-Potrie}.

\begin{prop}\label{ph-cone-criterion}
$f$ is a partially hyperbolic endomorphism if, and  only if,
there exists an unstable cone-field $\mathscr{C}^u$ on $M$. 
\end{prop}

\begin{rmk}
Partial hyperbolicity as well as dominated splitting is an open property. In other words, the existence of unstable cone-fields is a property shared by all nearby endomorphisms. Recalling that  $\mathscr{C}^u$ is an unstable cone-fields for $f$ if it
is invariant, transversal to the kernel and satisfy the unstable property, that is, there are $\ell, \lambda>1$ such that $ \|Df^{\ell}_x(v)\|\geq \lambda\|v\|,$  for all $x\in M$ and $v \in \mathscr{C}^u(x),$ see 
properties $[\mathrm{PH1}]$ and $[\mathrm{PH2}]$ in Section \cref{intro}.
\end{rmk}

\section{Proof of Main Theorem}\label{sec:main-thm}

In this section we prove the \hyperref[main-theo]{Main Theorem} stated in Section \Cref{intro}.
Since the existence of an invariant cone-field and a dominated splitting are equivalent, see Proposition \ref{cone-criterion} in Section \Cref{section-ph},
the \hyperref[main-theo]{Main Theorem} can be proved using the notion of dominated splitting in terms of invariant splitting. Before starting the proof let us construct precisely the invariant splitting over the set $\Lambda_f$ in \eqref{Lambda-set} that was previously defined
in Section \cref{intro}.


\subsection{Construction of the invariant splitting}\label{construction-lambda}
Let $f$ be a transitive endomorphism with $\mathrm{int}(\mathrm{Cr}(f))\neq \emptyset$ and $|\ker Df^n|= 1$ for every $n\geq 1$. Let $M_f$ be the inverse limit space for $f$, recalling that  $M_f=\{(x_i)_i \in M^{\mathbb{Z}}: f(x_i)=x_{i+1}, \, \forall i\in \mathbb{Z} \}$. It is known that the action of $f$ over $M_f$ (the shift map on $M_f$) is a transitive homeomorphism, for further details see \cite[Theorem 3.5.3]{AH},  and $\pi_0^{-1}(\mathrm{Cr}(f))$ has nonempty interior. Consequently, there is  a residual set of points in $M_f$ so that the backward/forward orbits are dense and intersect $\pi_0^{-1}(\mathrm{Cr}(f))$ infinitely many times for the past and for the future. Therefore, $\Lambda_f$ is a dense subset of $M_f$,  where  $\Lambda_f$ is defined by \eqref{Lambda-set}  in Section \cref{intro}.  From now on, we refer to this set as ``lambda-set" to avoid explicit reference to the endomorphism.


Let us recall  the candidate for the dominated splitting  for $f$ over $\Lambda_f$ introduced in Section~\cref{intro}. First, given $(x_i)_i \in \Lambda_f$, we denote by $\tau^{-}_{j}$ and $\tau^{+}_j,$ the first time before its orbit leave the set of critical points and the first time to enter in the set of critical points, respectively. More precisely, for every $(x_i)_i \in \Lambda_f$,
\begin{align}\label{n-past-future}
\begin{split}
\tau^{-}_{j}&=\max\{i<0: x_{j+i} \in \mathrm{Cr}(f) \,\, \text{and} \ \ f(x_{j+i}) \notin \mathrm{Cr}(f)\}; \ \  \text{and} \\
\tau^{+}_{j}&=\min\{i\geq 0:x_{j+i} \in \mathrm{Cr}(f)\}.
\end{split}
\end{align}
Thus, we define the subbundles of $TM_f$ over $\Lambda_f$ as:
\begin{align}\label{EF-splitting}\tag{$\ast \ast$}
E(x_j)=\ker \Big(Df^{^{\tau^{+}_j+1}}_{_{x_j}}\Big)
 \ \ \text{and} \ \ F(x_j)=
\mathrm{Im}\Big(Df^{^{|\tau^{-}_j|}}_{_{x_{\tiny{\tau^{-}_j}}}}\Big).
\end{align}
Since   $|\ker Df^n|= 1$ for all $n\geq 1$, we have that $E$ and $F$ are one-dimensional subbundles. Furthermore, it should be noted that if $x_j$ in $(x_i)_i$ is a critical point, then we have that $E(x_j)=\ker(Df_{x_j})$. In general, $E$ and $F$ at $x_j$ can be thought as the kernel of the $\tau^{+}_j$-iterated of $f$ where $\tau^{+}_j$ is the time that $f^i(x_j)$ takes to enter in $\mathrm{Cr}(f)$ and the image of the $|\tau^{-}_j|$-iterated of $f$ where $|\tau^{-}_j|$ is the time that $f^{-i}(x_j)$ 
takes to come back to $\mathrm{Cr}(f)$, respectively. 

\begin{figure}[!h]
\centering
\includegraphics[scale=0.7]{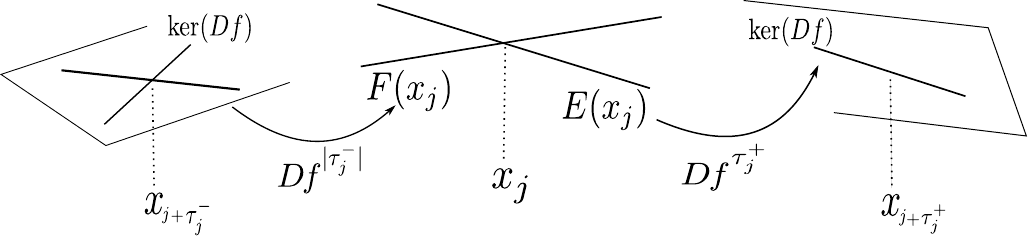}
\caption{$E,F$ subbundles over $\Lambda_f$.} \label{EF-splitting-picture}
\end{figure}


\begin{rmk}
It is interesting to notice that from the definition of $E(x_j)$ follows that it just depends on the forward orbit of $(x_{j+i})_{i\geq 0}$. However, $F(x_j)$ actually depends on the backward orbit of $(x_{j+i})_{i\leq 0}$.
\end{rmk}


We would like to emphasize that there is no assumption over the periodic points. The assumptions  $\mathrm{int}(\mathrm{Cr}(f))\neq \emptyset$ and $f$  transitive are to ensure that $\Lambda_f$ is a nonempty dense subset of $M_f$. Hence, we state the following result.

\begin{thm}\label{key-thm}
Let $f \in \mathrm{End}^1(M)$ be an endomorphism whose $|\ker(Df^n)|=1$ for every $n\geq 1$ and $\Lambda_f \neq \emptyset$. Then
\begin{itemize}
\item either $E\oplus F$ is a dominated splitting for $f$ over $\Lambda_f$;
\item or $f$ can be approximated by endomorphisms having full-dimensional kernel.
\end{itemize}
\end{thm}
Assuming the previous theorem, we are able to prove  \hyperref[main-theo]{Main Theorem}.

\begin{proof}[Proof of  Main Theorem]
Let us assume that $f$ cannot be approximated by endomorphisms having full-dimensional kernel. That is, there exists a neighborhood $\mathcal{U}_f$ of $f$ satisfying
\begin{align}\label{eq-assumption}\tag{$\star$}
\forall g \in \mathcal{U}_f,\, n\geq 1, \ \ |\ker(Dg^n)|=\max_{x \in M}|\ker(Dg^n_x)|\leq 1.
\end{align}
Then, by Theorem \ref{key-thm}, we have that $T_{\Lambda_f}M_f=E\oplus F$ is a dominated splitting for $f$. Moreover, using that $\Lambda_f$ is dense in $M_f$ together with Proposition \ref{continuity&extension}, we can extend the dominated splitting to the whole $M_f$, establishing the proof of  \hyperref[main-theo]{Main Theorem}. 
\end{proof}




\subsection{Proof of Theorem \ref{key-thm}} \label{pf-key-thm}
We start proving the invariance splitting property. 


\begin{lemma}[Invariance splitting]\label{invariance}
For every $(x_i)_i \in \Lambda_f$ holds,
\begin{align}\label{EF-invariant}\tag{\,I\,}
\begin{split}
& Df(E(x_i))\subseteq  E(x_{i+1}) \ \ \text{and} \ \ Df(F(x_i))=F(x_{i+1});\\
& \text{and} \ \  \bigsqcup_{i=-\infty}^{\infty} T_{x_i}M=\bigsqcup_{i=-\infty}^{\infty} E(x_i)\oplus F(x_i).
\end{split}
\end{align}
\end{lemma}

\begin{proof}
The proof follows immediately from the definition of $E$ and $F$ together with the fact that $Df_{x_j}$ is an isomorphism for $x_j \notin \mathrm{Cr}(f)$ and $\tau^\pm_{j+1}=\tau^{\pm}_j-1$.
\end{proof}

We would like to emphasize that the Lemma \ref{invariance} (Invariance splitting) above holds for any endomorphism $g \in \mathrm{End}^1(M)$ whose $|\ker(Dg^n)|=1$ for all $n \geq 1$ and $\Lambda_g$ the  nonempty set associated to $g$ given as in \eqref{Lambda-set}. 

\smallskip

Before proceeding, let us recall the following classical tool in $C^1$-perturbative arguments due to Franks in \cite{Franks}.

\begin{lemma*}[Franks]\label{Franks-lemma}
Let $f \in \mathrm{End}^1(M)$. Given a neighborhood $\mathcal{U}$ of $f$, there exist $\varepsilon >0$ and a neighborhood $\mathcal{U}'$ of $f$ contained in $\mathcal{U}$ so that given $g  \in \mathcal{U}'$, every finite collection of points $\Sigma=\{x_0,...,x_n\}$ in $M$, and every family of linear maps $L_i:T_{x_i}M \to T_{g(x_i)}M$ which $\|L_i-Dg_{x_i}\|<\varepsilon$, for each $0\leq i \leq n$, there exists an endomorphism $\tilde{g} \in \mathcal{U}$ and a family of balls $B_i$ centered in $x_i$ contained in a neighborhood $B$ of $\Sigma$ satisfying:
$\tilde{g}(x)=g(x)$, for $x \in M\backslash B$; $\tilde{g}(x_i)=g(x_i)$, and $\tilde{g}\mid_{B_i}=L_i$, for each $0\leq i \leq n$.
\end{lemma*}

Recalling that $\tilde{g}\mid_{B_i}=L_i$ means that the action of $\tilde{g}$ in $B_i$ is equal to the linear map $L_i$. 



The following result will be an important tool along this paper and will be used several times. Before state our assertion, let us denote by $E_g$ and $F_g$
the subbundles associated to $g$.

\begin{lemma}\label{Franks-conseq} 
For every neighborhood $\mathcal{U}$ of $f$, there exist $\varepsilon>0$ and a neighborhood $\mathcal{U}' \subseteq \mathcal{U}$ of $f$ such that for any $g \in \mathcal{U}'$ whose $\Lambda_g \neq \emptyset$ one has that if for some $(x_i)_i \in \Lambda_g$, there exist a subset $\Gamma=\{x_0,...,x_{n-1}\}$ and $L_i: T_{x_{i-1}}M \to T_{x_i}M$ linear maps satisfying:
$$\|L_i-Dg_{x_{i-1}}\| <\varepsilon, \ \ \text{for each} \ \ 1\leq i \leq n; \ \ \text{and} \ \  L_n \cdots L_1(F_g(x_0))=E_g(x_n).$$
Then, there exist $\tilde{g} \in \mathcal{U}$ and a neighborhood $W$ of $\Gamma$ such that:
\begin{enumerate}[label=$(\roman*)$]
\item $\tilde{g}(x_{i-1})=x_i$ and $D\tilde{g}_{x_{i-1}}=L_i$, for each $1\leq i \leq n$; and $\tilde{g}\mid_{W^c}=g\mid_{W^c}$;
\item there exists $m \geq 1$ such that $|\ker(D\tilde{g}^m)|=2$.
\end{enumerate}
\end{lemma}

\begin{proof} 
The existence of $\tilde{g}$ is guaranteed by \hyperref[Franks-lemma]{Franks' Lemma}. Assume, without loss of generality, that $\tau^{-}_0<n \leq \tau^{+}_0$ where $x_{\tau^{-}_0}$ and $x_{\tau^{+}_0}$ belong to $\mathrm{Cr}(g)$ and $$W\cap\{x_{\tau^{-}_0},\dots,x_0,\dots,x_{n-1},\dots,x_{\tau^{+}_0}\}=\Gamma.$$
Then, taking $m=1+|\tau^{-}_0|+\tau^{+}_0$, one has
\begin{align*}
D\tilde{g}^m(T_{x_{\tau^{-}_0}}M)&=Dg^{1+\tau^{+}_0-n} L_n \cdots L_1 Dg^{|\tau^{-}_0|}(F(x_{\tau^{-}_0}))\\
&=Dg^{1+\tau^{+}_0-n} L_n \cdots   L_1(F(x_0))=Dg^{1+\tau^{+}_0-n}(E(x_n))=\{0\}.
\end{align*}
In other words, we obtain that $|\ker(D\tilde{f}^m)|=2$.
\end{proof}

It remains to prove that the angle between $E$ and $F$ is uniformly bounded away from zero and there exists $\ell\geq 1$ so that $E \prec_{\ell} F$, recall Definition \ref{def-DS}. In order to prove it, we assume that $f$ cannot be approximated by endomorphisms having full-dimensional kernel, 
that is, there exists a neighborhood $\mathcal{U}_f$ of $f$ satisfying \eqref{eq-assumption}.

Throughout this section, we fix $\varepsilon>0$ and a neighborhood $\mathcal{U}$ of $f$ contained in $\mathcal{U}_f$ such that the hypotheses of
Lemma \ref{Franks-conseq}  are never satisfied.
Further, let us fix an angle $\alpha>0$ small enough such that any rotation map $R$ of angle less than $\alpha$  verifies  that 
$\|R \circ Dg - Dg\|<\varepsilon$, for all $g \in \mathcal{U}$. By slight abuse of notation, 
we continue denoting $\mathcal{U}$ by $\mathcal{U}_f$. 



The following lemma ensures that the angle between $E$ and $F$ is uniformly bounded away from zero.

\begin{lemma}[Uniform angle]\label{unif-angle}
For every $g \in \mathcal{U}_f$ with $\Lambda_g \neq \emptyset$, we have for every $(x_i)_i \in \Lambda_g$ that
\begin{align}\tag{II}
\varangle (E_g(x_i),F_g(x_i)) \geq \alpha,\ \ \text{for all} \ \ i \in \mathbb{Z}.
\end{align}
\end{lemma}

\begin{proof} 
Suppose that there exists a sequence of points $x_i$ such that the angles between $E_g(x_i)$ and $F_g(x_i)$ is less than $\alpha$. Let $R:T_{x_i}M \to T_{x_i}M$ be a rotation of angle less than $\alpha$ such that $R(F(x_i))=E(x_i)$. Then, $L_i=R \circ Dg$ satisfies
$$\|Dg - L_i \|\leq \varepsilon \ \ \text{and} \ \ L_i(F_g(x_{i-1}))=E_g(x_i),$$
which is a contradiction. Therefore, for every $(x_i)_i \in \Lambda_g$ and $i \in \mathbb{Z}$, $$\varangle (E_g(x_i),F_g(x_i)) \geq \alpha.$$
\end{proof}



Note that so far we have proved that for all $g \in \mathcal{U}_f$ with $\Lambda_g$ nonempty, the splitting $E_g\oplus F_g$ over $\Lambda_g$ is invariant and its angle is uniform bounded away from zero. Finally,  let us prove that the invariant splitting constructed above has the domination property.

\begin{lemma}[Uniform domination property] \label{domination-openess}
There exists a neighborhood $\mathcal{U}$ of $f$ contained in  $\mathcal{U}_f$ and an integer $\ell >0$ such that for every $g \in \mathcal{U}$, whose $\Lambda_g$ is nonempty, we have for every $(x_i)_i \in \Lambda_g$ that:
\begin{align}\tag{III}
\|Dg^{\ell}\mid_{E_g(x_i)}\|\leq \frac{1}{2}\|Dg^{\ell}\mid_{F_g(x_i)}\|, \ \ \text{for all} \ \ i\in \mathbb{Z}.
\end{align}
\end{lemma}

The proof of the uniform domination property will be done in the next subsection. The lemma above is slightly more general than the domination property, indeed the statement claims that the angle is bounded away from zero, and the domination property are uniform in an open set for endomorphisms whose ``lambda-set" is nonempty. This will be useful when we wish to prove that the accumulation point of $g \in \mathcal{U}$ with $\Lambda_g$ nonempty, has the domination property as well. Therefore, we conclude the proof of Theorem \ref{key-thm}. \qed  

\subsection{Uniform domination property}\label{uniform-DP}
This section is devoted to prove the Lemma \ref{domination-openess} (Uniform domination property) stated in previous section.
Let $f$ be an endomorphism displaying critical points satisfying the assumption \eqref{eq-assumption}. 



Let us highlight that for proving this lemma will not be assumed that $\Lambda_f$ is nonempty.
That is, in order to obtain the uniform domination property,  $\Lambda_f$ nonempty is not a necessary condition.
The proof of the uniform dominated property will be divided in two steps.
 The first one shows the domination property in a neighborhood of the critical points.
 In the second step, we extend the domination property got in the first step  to the ``lambda-set'' to conclude the lemma.



\subsubsection*{\uline{First step}}\label{pf-domination-critic} 



Domination property nearby the set of critical points.

\begin{lemma}\label{domination-open} 
There exist two neighborhoods $\mathcal{U}'\subseteq \mathcal{U}_f$ of $f$ and $\mathrm{U}$ of $\mathrm{Cr}(f)$ such that for every $g \in \mathcal{U}'$ which $\Lambda_g$ is well-defined, we have that if  $(x_i)_i \in \Lambda_g$ satisfies
\begin{align}\label{eq-lemma}
\|Dg^j\mid_{E_g(x_0)}\|\geq \frac{1}{2}\|Dg^j\mid_{F_g(x_0)}\|,
\end{align}
for every $1\leq j \leq l$. Then, $x_0,x_1,\dots, x_{l-1} \notin \mathrm{U}$.
\end{lemma}

The proof requires some preliminaries notation and results.

\medskip

Since $\min_{\|v\|=1}\|Df_x(v)\|< \max_{\|w\|=1}\|Df_x(w)\|$ for every $x$ nearby $\mathrm{Cr}(f)$, we define $V_f(x)$ and $W_f(x)$ as the subspaces of $T_xM$ in a neighborhood $ \mathrm{U}$ of $\mathrm{Cr}(f)$ verifying:
\begin{align}\label{new-split}
\|Df\mid_{V_f(x)}\|=\min_{\|u\|=1}\|Df_x(u)\| \ \ \text{and} \ \ 
\|Df\mid_{W_f(x)}\|=\max_{\|w\|=1}\|Df_x(w)\|.
\end{align}
More general, we can assume that for every $g \in \mathcal{U}_f$ one has in $\mathrm{U}$,
$$\min_{\|v\|=1}\|Dg_x(v)\|< \max_{\|w\|=1}\|Dg_x(w)\|;$$
and define $V_g(x)$ and $W_g(x)$ for $x \in \mathrm{U}$, similarly. Furthermore, whenever $\mathrm{Cr}(g) \neq \emptyset,$  $\mathrm{Cr}(g)$ is contained in $\mathrm{U}$, that is, $\mathrm{U}$ is a neighborhood of the critical points for every $g \in \mathcal{U}_f$, and   $V_g(x)=\ker(Dg_x)$ for all $x \in \mathrm{Cr}(g)$. On the other hand, note that the following functions: 
\begin{align}\label{continuity-VW}
\mathcal{U}_f \ni g \mapsto V_g, \, W_g, \ \ \text{and} \ \ \mathrm{U} \ni x \mapsto V_g(x), \, W_g(x)
\end{align}
are continuous. Moreover, it is well known that $V_g(x)$ and $W_g(x)$ are orthogonal for all $x \in \mathrm{U}$. 

Now consider the cone-field $\mathscr{C}_{V_g}$ with core $V_g$ and angle $\eta>0$,
$$\mathscr{C}_{V_g}: \mathrm{U} \ni x \mapsto \mathscr{C}_{V_g}(x,\eta)=\{v+w \in V_g(x)\oplus W_g(x): \|w\|\leq \eta \|v\|\};$$
and recall that $\mathscr{C}_{V_g}^{\ast}(x,\eta)$ is the dual cone-field of $\mathscr{C}_{V_g}(x,\eta)$ defined as the closure of $T_xM\backslash\mathscr{C}_{V_g}(x,\eta)$.

Next result gives an interesting property about the splitting $V_g\oplus W_g$ over $\mathrm{U}$.

\begin{prop}\label{new-splitting}
Given $\eta>0$ and $\theta>0,$ there are two neighborhoods $\mathcal{U} \subseteq \mathcal{U}_f$ of $f$ and $\mathrm{U}$ of the critical points for every $g \in \mathcal{U}$ so that for every $g \in \mathcal{U}$ and $x \in \mathrm{U}$ holds 
\begin{align}\label{VW-property}
\varangle(Dg(W_g(x)),\mathbb{R}\langle Dg_x(u)\rangle)<\theta, \, \forall\, u  \in \mathscr{C}^{*}_{V_g}(x,\eta).
\end{align}
\end{prop}

\begin{proof}
Since $Df(\mathscr{C}^{*}_{V_f}(x,\eta))=Df(W_f(x))$ for every $x \in \mathrm{Cr}(f)$ and by the equation \eqref{continuity-VW}, we can find a neighborhood $\mathcal{U}\subseteq \mathcal{U}_f$ of $f$ and a neighborhood $\mathrm{U}$ of the critical points for every $g \in \mathcal{U}$ verifying  \eqref{VW-property} as we wished.
\end{proof}

Next lemma provides the relation between the splitting $E_g\oplus F_g$ and $V_g\oplus W_g$.
Given $\eta>0$ and $\theta >0$ small enough,
consider the  neighborhoods $\mathcal{U}$ and $\mathrm{U}$ given in Proposition \ref{new-splitting}.

\begin{lemma}\label{cont-nearby-sing}
For every $g \in \mathcal{U}$ with $\Lambda_g$ well-defined and for every $(x_i)_i \in \Lambda_g$ with $x_0=x \in \mathrm{U}$ holds
\begin{align}\label{eq-cone}
F_g(x) \subseteq \mathscr{C}_{V_g}^{\ast}(x,\eta) \ \ \text{and} \ \ E_g(x) \subseteq \mathscr{C}_{V_g}(x,\eta).
\end{align}
In particular, it follows that $\varangle(Dg_x(W_g(x)),F_g(g(x)))<\theta$.

\end{lemma}

\begin{proof}
Without loss of generality, we can assume  $\eta >0$ and $\theta>0$ small enough 
such that  $0<\eta< \alpha$, where $\alpha$ is given by the Lemma \ref{unif-angle} (Uniform angle),
implying that for every $g \in \mathcal{U}_f$ with $\Lambda_g$ nonempty has  the angle between $E_g$ and $F_g$ greater than $\alpha$, 
and for every rotation $R$ of angle less than $2\eta$ or $2\theta$ holds for every $g \in \mathcal{U}$,
\begin{align}\label{ineq:rotation}
\|R\circ Dg -Dg \|< \varepsilon.
\end{align}

\smallskip

In order to prove  \eqref{eq-cone} we will show that:
\begin{enumerate}[label=(\alph*)]
\item $E_g(x)$ and $F_g(x)$ cannot be contained simultaneously  in neither $\mathscr{C}_{V_g}(x,\eta)$ nor $\mathscr{C}_{V_g}^{\ast}(x,\eta)$;
\item $F_g(x)$ and $E_g(x)$ cannot be contained in $\mathscr{C}_{V_g}(x,\eta)$ and $\mathscr{C}_{V_g}^{\ast}(x,\eta)$, respectively.
\end{enumerate}

\medskip

Let us prove item (a). As the angle between $E_g$ and $F_g$ is greater than $\alpha$, one concludes that both $E_g(x)$ and $F_g(x)$ cannot be contained in $\mathscr{C}_{V_g}(x,\eta)$. On the other hand, if $E_g(x)$ and $F_g(x)$ are contained in $\mathscr{C}_{V_g}^{\ast}(x,\eta)$,
by \eqref{VW-property} we have that
\begin{align*}
\varangle(E_g(g(x)),F_g(g(x)))\leq \varangle(E_g(g(x)),Dg(W_g(x))) +\varangle(Dg(W_g(x)), F_g(g(x)))\leq 2\theta. 
\end{align*}
Then, we can take a rotation $R:T_{g(x)}M \to T_{g(x)}M$ such that $$R(F_g(g(x)))=E_g(g(x)) \ \ \text{and} \ \ \|R\circ Dg_x-Dg_x\|<\varepsilon$$
which implies, by Lemma \ref{Franks-conseq}, that some perturbation of $g$ has full-dimensional kernel which is a contradiction. This proves (a).

To prove (b), we suppose that $F_g(x)\subseteq \mathscr{C}_{V_g}(x,\eta)$ and $E_g(x) \subseteq \mathscr{C}_{V_g}^{\ast}(x,\eta)$. Then, we can take $v \in \mathscr{C}_{V_g}^{\ast}(x,\eta)$ such that $\varangle(\mathbb{R}\langle v\rangle, F_g(x))<2\eta$ and
\begin{align*}
\varangle(\mathbb{R}\langle Dg_x(v)\rangle, E_g(g(x)))&\leq \varangle(\mathbb{R}\langle Dg_x(v)\rangle, Dg(W_g(x)))\\ &+\varangle(Dg(W_g(x)), E_g(g(x))) \leq 2\theta,
\end{align*}
and so, there exist two rotations $R_0:T_xM \to T_xM$ and $R_1:T_{g(x)}M \to T_{g(x)}M$ such that
$R_0(F_g(x))=\mathbb{R}\langle v\rangle$ and $R_1(\mathbb{R}\langle Dg_x(v)\rangle)=E_g(g(x))$ with $\|R_{\ast}\circ Dg-Dg\|<\varepsilon$, for $\ast=0,1$, which contradicts the assumption (\ref{eq-assumption}). Therefore, we conclude the proof of (b), and in  consequence, of the lemma. 
\end{proof}

An immediate consequence of Lemma \ref{cont-nearby-sing} is the following.

\begin{lemma}\label{constant-c0}
Given $\nu >0$ and $\eta >0$ small enough, there exist neighborhoods $\mathcal{U}_f$ of $f$ and $\mathrm{U}$ of $\mathrm{Cr}(f)$ such that for every $g \in \mathcal{U}_f$ whose $\Lambda_g$ is nonempty, one has that for $(x_i)_i \in \Lambda_g$ with $x_0=x \in \mathrm{U}$,
\begin{align}
\|Dg\mid_{E_g(x)}\|< \nu \ \ \text{and} \ \ \|Dg_x(v)\|\geq 2\nu,\, \forall\, v \in \mathscr{C}_{V_g}^{\ast}(x,\eta).
\end{align}
\end{lemma}

Finally, we can prove Lemma \ref{domination-open}.

\begin{proof}[Proof of Lemma \ref{domination-open}]
It follows from Lemma \ref{cont-nearby-sing} that given $\eta>0$ and $\theta>0$, there are neighborhoods $\mathcal{U}$ of $f$ contained in $\mathcal{U}_f$ and $\mathrm{U}$ of $\mathrm{Cr}(f)$ so that for every $g \in \mathcal{U}$ whose $\Lambda_g$ is nonempty and $(x_i)_i \in \Lambda_g$ with $x_0=x \in \mathrm{U}$, one has that:
\begin{align*}
E_g(x) \subseteq \mathscr{C}_{V_g}(x,\eta) \ \ \text{and} \ \ F_g(x) \subseteq \mathscr{C}_{V_g}^{\ast}(x,\eta).
\end{align*}
In particular, for all $u \in \mathscr{C}_{V_g}^{\ast}(x,\eta)$ we have that $\varangle(\mathbb{R}\langle Dg(u)\rangle, Dg(W_g(x)))<\theta$.

For $(x_i)_i \in \Lambda_g$, we define the cone-field $\mathscr{C}_{E_g,F_g}(x_i,\beta)$ with core $F_g$ and length $\beta>0$ in coordinates $E_g\oplus F_g$ over $(x_i)_i$ by:
$$\mathscr{C}_{E_g,F_g}(x_i,\beta)=\{u_1+u_2\in E_g(x_i)\oplus F_g(x_i):\|u_1\|\leq \beta \|u_2\|\}.$$

Note that the cone-field above is in $E_g\oplus F_g$ coordinates which may not be an orthogonal splitting. However, since the angle between $E_g$ and $F_g$ is uniformly bounded away from zero, we may suppose that  $\eta, \theta >0$ are chosen small enough so that for every $(x_i)_i\in \Lambda_g$ whose $x_0 \in \mathrm{U}$ follows that:
\begin{itemize}
\item[-] for every $v \in \mathscr{C}_{V_g}(x_0,\eta)$ holds $\varangle(\mathbb{R}\langle v \rangle, E_g(x_0))<\beta$;\smallskip
\item[-] for all $u \in T_{x_0}M$ holds:
$$\varangle(\mathbb{R}\langle Dg_{x_0}(u)\rangle,F_g(g(x_0)))<2\theta \Longrightarrow Dg_{x_0}(u) \in \mathscr{C}_{E_g,F_g}(g(x_0),\beta/2).$$
\end{itemize}

Now, assume that $(x_i)_i \in \Lambda_g$ is any point satisfying the equation \eqref{eq-lemma}. That is, $(x_i)_i$ satisfies:
$$\|Dg^j\mid_{E_g(x_0)}\| \geq \frac{1}{2} \|Dg^j\mid_{F_g(x_0)}\|, \, \text{for each} \ \ 1 \leq j \leq l.$$
Hence, Lemma \ref{constant-c0} implies $x_0 \notin \mathrm{U}$. Thus, it remains to show that $x_1,\dots,x_{l-1} \notin \mathrm{U}$. Suppose, without loss of generality, that $x_{l-1} \in \mathrm{U}$. Let $u=u_1+u_2 \in E_g(x_0)\oplus F_g(x_0)$ a vector in the closure of $TM\backslash \mathscr{C}_{E_g,F_g}(x_0,\beta)$, denoted by $\mathscr{C}_{E_g,F_g}^{\ast}(x_{0},\beta)$. That is, $\|u_1\|\geq \beta \|u_2\|$. Then, $Dg^j(u)=Dg^j(u_1)+Dg^j(u_2) \in E_g(x_j)\oplus F(x_j)$ satisfying
\begin{align}\label{eq-final}
\|Dg^j(u_1)\|=\|Dg^j\mid_{E_g(x_0)}\|\|u_1\|\geq \frac{\beta}{2} \|Dg^j\mid_{F_g(x_0)}\|\|u_2\|=\frac{\beta}{2}\|Dg^j(u_2)\|.
\end{align}
In other words, $Dg^j(u)$ does not belong to $\mathscr{C}_{E_g,F_g}(x_j,\beta/2)$ for $1\leq j \leq l$. On the other hand, for every $w \in \mathscr{C}_{V_g}^{\ast}(x_{l-1},\eta)$,
\begin{align*}
\varangle(\mathbb{R}\langle Dg(w)\rangle, F_g(x_{l}))&\leq \varangle(\mathbb{R}\langle Dg(w)\rangle, Dg(W_g(x_{l-1})))\\& + \varangle(Dg(W_g(x_{l-1})), F_g(x_{l}))\leq 2\theta,
\end{align*}
and so $Dg^{l-1}(u)$ does not belong to $\mathscr{C}_{E_g,F_g}^{\ast}(x_{l-1},\eta)$ which implies that $Dg^{l-1}(u) \in \mathscr{C}_{V_g}(x_{l-1},\eta),$ and consequently, $ \varangle(\mathbb{R}\langle Dg^{l-1}(u) \rangle, E_g(x_{l-1}))<\beta$. Finally, assuming that $\beta >0$ is small enough so that $u \in \mathscr{C}_{E_g,F_g}(x_0,\beta)$ implies $\varangle(\mathbb{R}\langle u\rangle,F_g(x_0))$  is small enough, 
there are two rotations $R_1$ and $R_2$ on $T_{x_0}M$ and $T_{x_{l-1}}M$, respectively, such that 
$R_1(F_g(x_0))=\mathbb{R}\langle u \rangle,$ $R_2(\mathbb{R}\langle Dg^{l-1}(u) \rangle)=E(x_{l-1})$, and 
$\|R_{\ast}\,\circ Dg-Dg\|<\varepsilon$, for $\ast=1,2$.
However, by Lemma \ref{Franks-conseq} and assumption (\ref{eq-assumption}), we get a contradiction. This concludes the proof.
\end{proof}



\subsubsection*{\uline{Second step}} Domination property on the ``lambda-set". 
\smallskip

In order to finish the proof of the Lemma \ref{domination-openess} (Uniform domination property), we use the first step to extend the domination property to the corresponding ``lambda-set" for every nearby endomorphism. First let us introduce the following auxiliary result, the proof will be omitted and can be found in \cite[Appendix A]{Potrie}.

\begin{lemma}\label{Potriethesis}
Given $\kappa>0$ and $K>0$, there exists $l>0$ such that if $A_1,...,A_l$ is a sequence in $GL(2,\R)$ and $v, w$ are unit vectors in $\mathbb{R}^2$ verifying:
\begin{align}
\max_{1\leq i \leq l}\{\|A_i\|,\|A^{-1}_i\|\}\leq K; \ \ \text{and} \ \  \|A_l...A_1(v)\| \geq \frac{1}{2} \|A_l...A_1(w)\|.
\end{align}
Then, there exist rotations $R_1,...,R_l$ of angles less than $\kappa$ such that $$R_lA_l...R_1A_1(\R\langle w \rangle)=A_l...A_1(\R\langle v \rangle).$$
\end{lemma}

Let $\mathcal{U}$ and $\mathrm{U}$ given by Lemma \ref{domination-open} in the first step. We now prove the following.

\medskip

\noindent
\textbf{Claim 1}: There exists $l\geq 1$ such that for every $(x_i)_i \in \Lambda_g$ there exists an integer $j:=j((x_i)_i), 1 \leq j \leq l,$ such that,
\begin{align}\label{eqclaim1}
\|Dg^j\mid_{E_g(x_0)}\|\leq \frac{1}{2}\|Dg^j\mid_{F_g(x_0)}\|.
\end{align}

\begin{proof}[Proof of Claim 1]
For every $g \in \mathcal{U},$ either $\mathrm{Cr}(g)$ is empty or $\mathrm{Cr}(g)$ is contained in $\mathrm{U}$. Then, there are $K\geq \max\{\|Dg_y\|,\|Dg^{-1}_y\|\}$ uniformly in  $M\backslash \mathrm{U},$ for every $g \in \mathcal{U},$ and $\kappa>0$ small enough such that any rotation $R$ of angle less than $\kappa$ satisfies: $$\|R\circ Dg -Dg\|<\varepsilon, \ \ \text{for all} \ \ g \in \mathcal{U}.$$
Fix $l_1\geq 1$ as in Lemma \ref{Potriethesis}. Suppose now that  (\ref{eqclaim1}) does not hold. In particular, for $l \geq l_1$ there exists $(x_i)_i \in \Lambda_g$ such that
\begin{align}
\|Dg^j\mid_{E_g(x_0)}\|\geq \frac{1}{2}\|Dg^j\mid_{F_g(x_0)}\|.
\end{align}
for every $1\leq j \leq l$. By Lemma \ref{domination-open}, we have that $x_0, x_1, \dots, x_{l-1} \notin \mathrm{U}$. And so, for each $1 \leq j \leq l,$ $A_j=Dg_{x_{j-1}}$  verifies Lemma \ref{Potriethesis}. Consequently, there exist rotations $R_1,\dots,R_{l}$ so that:
$$ \|R_jA_j-A_j\|<\varepsilon, \,  1 \leq j \leq l; \ \ \text{and} \ \ R_lA_l\cdots R_1A_1(F_g(x))=E_g(g^l(x)).$$
By Lemma \ref{Franks-conseq} and the assumption \eqref{eq-assumption}, we get a contradiction. This concludes the proof of Claim 1.
\end{proof}

Fix $l_1\geq 1$ given by Lemma \ref{Potriethesis}. It satisfies Claim 1, that is, for every $g \in \mathcal{U}$ whose $\Lambda_g$ is nonempty such  that for each $(x_i)_i \in \Lambda_g$, there exists an integer $j_0:=j((x_i)_i), \, 1\leq j \leq l_1$, verifying:
\begin{align}\label{eq-key}
\|Dg^{j_0}\mid_{E_g(x_0)}\|\leq \frac{1}{2}\|Dg^{j_0}\mid_{F_g(x_0)}\|.
\end{align}

Finally, in order to conclude the proof of the  Lemma \ref{domination-openess} (Uniform domination property) remains to show the following assertion.

\medskip
\noindent
\textbf{Claim 2}: There exists $\ell \geq 1$ such that for every $(x_i)_i \in \Lambda_g$,
\begin{align*} 
\|Dg^{\ell}\mid_{E_g(x_0)}\|\leq \frac{1}{2}\|Dg^{\ell}\mid_{F_g(x_0)}\|.
\end{align*}

We emphasize that $\ell$ is uniform in $\mathcal{U}$. 

\begin{proof}[Proof of Claim 2]
Assume $\ell$ greater than $l_1$. Let $j_0=j((x_i)_i),\, 1 \leq 1 \leq l_1,$ such that:
\begin{align*}
\|Dg^{\ell}\mid_{E_g(x_0)}\|&=\|Dg^{\ell-j_0}\mid_{E_g(x_{j_0})}\|\|Dg^{j_0}\mid_{E_g(x_0)}\|\\
&\leq \frac{1}{2}\|Dg^{\ell-j_0}\mid_{E_g(x_{j_0})}\|
\|Dg^{j_0}\mid_{F_g(x_0)}\|.
\end{align*}
Then, repeating $r$-times the process, we obtain $L_{r}=\ell-\sum_{i=0}^{r-1} j_i$ with $1\leq L_r\leq l_1$ such that:
\begin{align*}
\|Dg^{\ell}\mid_{E_g(x_0)}\| &\leq  \biggl(\frac{1}{2}\biggr)^r\|Dg^{L_r}\mid_{E_g(x_{j_{r-1}})}\|
\|Dg^{\ell-L_r}\mid_{F_g(x_0)}\| \\
& \leq  \biggl(\frac{1}{2}\biggr)^r C_0\|Dg^{\ell}\mid_{F_g(x_0)}\|,
\end{align*}
where $C_0$ is such that:
\begin{align}\label{c0}
 C_0\geq \frac{\max_{x \in M}\{\|Dg^i_x\|:i=1,2, \dots ,l_1\}}
 {\min\{2\nu, \min_{\|v\|=1}\{\|Dg(v)\|^{l_1}:x\in M\backslash \mathrm{U}\} \}}
\end{align}
with $\nu>0$ given by Lemma \ref{constant-c0}.

\begin{rmk}\label{contstan-rmk}
The constant $C_0$ is well-defined, since $F_g(x) \subseteq \mathscr{C}_{V_g}^{\ast}(x,\eta)$ in $\mathrm{U}$ and by the fact that $\min_{\|v\|=1}\|Dg_x(v)\|$ is positive in $M\backslash \mathrm{U}$. In particular, we can take $C_0$ uniform for the neighborhood  $\mathcal{U}$.
\end{rmk}

Therefore, taking $\ell\geq 1$ large enough such that $(1/2)^r C_0 \leq 1/2$, we have for each $(x_i)_i \in \Lambda$ that:

\smallskip
$\bullet$ for $\tau^{+}_0 \geq \ell$, 
$$\|Dg^{\ell}\mid_{E_g(x_0)}\|\leq \frac{1}{2}\|Dg^{\ell}\mid_{F_g(x_0)}\|;$$

$\bullet$ for $\tau^{+}_0 < \ell, \ \ \|Dg^{\ell}\mid_{E_g(x_0)}\|=0$. In particular,
$$\|Dg^{\ell}\mid_{E_g(x_0)}\|=0 \leq \frac{1}{2}\|Dg^{\ell}\mid_{F_g(x_0)}\|.$$
This concludes the proof of Claim 2 and, consequently, the Lemma \ref{domination-openess} (Uniform domination property). 
\end{proof}

\section{Consequences of Main Theorem}\label{sec:main-thm-consequences}

In this section we introduce an auxiliary result that will be used for proving Theorem \ref{thm-B} stated in Section \cref{intro}. Both results are proved in this section. Recalling that Theorem \ref{thm-B} asserts  that the torus and the Klein bottle are the only surfaces that admit robustly transitive endomorphisms. Concretely,

\begin{thm}\label{thm-DS}
If $f \in \mathrm{End}^1(M)$ is a robustly transitive endomorphism displaying critical points, then $M_f$ admits a dominated splitting for $f$.
\end{thm} 
 

Before proving the theorem above, we prove Theorem \ref{thm-B}.

\subsection{Proof of Theorem B}\label{pf-thm-B}
Assuming Theorem \ref{thm-DS}, denote by $E\oplus F$ the dominated splitting over $M_f$ for $f$. Thus, by Remark \ref{proj-E}, we have that $E$ induces a continuous subbundle on $TM$ over $M$, by slight abuse of notation it is denoted by $E$ as well. Let $(\widetilde{M}, p, p^{\ast}(E))$ be the double covering of $E$ over $M$. Hence, since the subbundle $p^{\ast}(E)$ of $T\widetilde{M}$ is orientable, we can define a vector field $X:\widetilde{M} \to T\widetilde{M}$ such that $X(x)\not=0 \in p^{\ast}(E)$. Therefore, one gets that $\chi(\widetilde{M})=0$, and so, $\chi(M)=0$. Thus, $M$ is either the torus $\T^2$ or  the Klein bottle $\K^2$. This complete the proof of Theorem B.\qed


\medskip

We now prove Theorem \ref{thm-DS}.

\subsection{Proof of Theorem $\ref{thm-DS}$}
Let $\mathcal{U}_f \subseteq \mathrm{End}^1(M)$ be a neighborhood of $f$ such that every $g \in \mathcal{U}_f$ is transitive. By 
the \hyperref[key obstruction]{Key Lemma}  stated in Section~\cref{intro}, we have that every $g \in \mathcal{U}_f$ satisfies the assumption \eqref{eq-assumption}, recalling that this assumption says that $|\ker(Dg^n)|\leq 1$, for all $g \in \mathcal{U}_f$ and $n\geq 1$. 


Define $\mathcal{D}$ as the subset of $\mathcal{U}_f$ given by,
\begin{align}
\mathcal{D}=\{g \in \mathcal{U}_f: \mathrm{int}(\mathrm{Cr}(g))\neq \emptyset\}.
\end{align}

Observe that since every $g \in \mathcal{D}$ is transitive and $\mathrm{int}(\mathrm{Cr}(g))\neq \emptyset$, 
then $\Lambda_g$ is dense in $M_g$, recall \eqref{Lambda-set} in Section \cref{intro} and \cref{construction-lambda}.
\begin{lemma}\label{typically}
For any $g \in \mathcal{U}_f$ whose $\mathrm{Cr}(g)\neq \emptyset$ and any neighborhood $\mathcal{U} \subseteq \mathcal{U}_f$ of $g$, holds that $\mathcal{D}\cap \mathcal{U}\neq \emptyset$. In particular, $\mathcal{D}$ contains a family of endomorphisms converging to $f$.
\end{lemma}

\begin{proof}
Suppose, without loss of generality, that $\mathrm{int}(\mathrm{Cr}(g))$ is empty and let $p$ in $\mathrm{Cr}(g)$. By \hyperref[Franks-lemma]{Franks' Lemma}, we can consider two sequences, the first one $(B_n)_n$ of neighborhoods of $p$ and the other one $(g_n)_n$ in $\mathcal{U}$ such that for all $n\geq 1$ holds:
\begin{enumerate}
 \item[-]  $\mathrm{int}(\mathrm{Cr}(g_n))\subseteq B_n$, $g_n(x)= Dg_p$ for $x\in \mathrm{int}(\mathrm{Cr}(g_n))$, and the diameter of $B_n$ goes to zero as $n$ goes to infinity;
 \item[-] $g_n$ converges to $g$ in ${\rm{End}}^1(M)$, and $\mathrm{int}(\mathrm{Cr}(g_n))\neq \emptyset$ for every $n$;
 \item[-] $g_n \mid_{M\backslash B_n}=g$ and $g_n(p)=g(p)$.
\end{enumerate}
This proves the lemma.
\end{proof}

Assume, up to changing the neighborhood of $f$, $\mathcal{U}_f$ satisfies 
the Lemma \ref{domination-openess} (Uniform domination property), recall Section~\cref{sec:main-thm}. Considering the family $(f_n)_n$ in $\mathcal{D}$ converging to $f$
given by Lemma \ref{typically}, the \hyperref[main-theo]{Main Theorem} implies that each $f_n$ admits a dominated splitting $E_n\oplus F_n$ verifying:
\begin{align}
\begin{split}
\exists\, \alpha >0, \,& \ell \geq 1 \ \ \text{such that} \ \ \forall (x_i)_i \in M_{f_n}, i \in \mathbb{Z},\, \text{and} \, \, n \geq 1, \\ &
\varangle(E_n(x_i),F_n(x_i))>\alpha \ \ \text{and} \ \ E_n\prec_{\ell} F_n.
\end{split}
\end{align}
On the other hand, by Lemma \ref{cone-criterion}, we can find a family of cone-fields $\mathscr{C}_{E_n}$ of uniform angle (length) such that $$Df_n^{\ell}(\mathscr{C}_{E_n}^{\ast}(x))\subseteq \mathrm{int}(\mathscr{C}_{E_n}^{\ast}(x)) \ \ \text{and} \ \ T_xM=E_n(x) \oplus \mathscr{C}_{E_n}^{\ast}(x),$$
for all $x \in M$. Furthermore, the uniqueness of the domination property implies that the following limits are well-defined,
\begin{align}
E(x)=\lim E_n(x) \ \ \text{and} \ \ \mathscr{C}_{E}^{\ast}(x)=\lim \mathscr{C}_{E_n}^{\ast}(x).
\end{align} 
And so, we have that for every $x \in M$,
$$Df^{\ell}(\mathscr{C}_{E}^{\ast}(x))\subseteq \mathrm{int}(\mathscr{C}_{E}^{\ast}(x)) \ \ \text{and} \ \ T_xM=E(x) \oplus \mathscr{C}_{E}^{\ast}(x).$$
Hence completes the proof of Theorem \ref{thm-DS}.
\qed

\section{Homotopy classes}\label{homotopy}


This section is devoted to the proof of Theorem \ref{thm-C}, Corollaries \ref{cor1} and \ref{cor2}, stated in   Section \cref{intro}. Let us recall using a revisited version of Theorem \ref{thm-C}.

\begin{thm-C} 
If $f \in \mathrm{End}^1(M)$ is a transitive endomorphism admitting a dominated splitting. Then, the action of $f$ on the first homology group has at least one eigenvalue with modulus larger than one.
\end{thm-C}

Throughout this section, we denote by $E$ a $Df$-invariant continuous subbundle of $TM$ and $\mathscr{C}$ a  continuous invariant cone-field on $M$ transverse to $E$, recalling that if $x \in \mathrm{Cr}(f)$ then $E(x)=\ker(Df_x)$. 
To prove Theorem \ref{thm-C},  we first show that the length of the iterates of an arc $\gamma$ tangent to the cone-field $\mathscr{C}$ by $f$ grows exponentially. Afterward, 
we find a relation between the $\gamma$ iterates growth and the volume growth of the balls containing the iterates of $\gamma$, allowing us  to conclude the proof of Theorem \ref{thm-C}. Furthermore, the exponential growth of the iterates by $f$ of a tangent arc to the cone-field is the first step in order to find an expanding direction for $f.$ This is proved in Section \cref{sec:exp direction}, recalling that one of our goals is to prove that $f$ is partially hyperbolic, hence completing the proof of Theorem \ref{thm-A}.


\subsection{Topological expanding direction}
Let us fix some notion that will be used along this section.
An {\it{u-arc}} is an injective Lipschitz curve $\gamma:[0,1]\to M$ such that $\gamma' \subseteq \mathscr{C}$, where $\gamma'$ denotes the set of  tangent vectors to $\gamma$. In other words, all tangent vectors to $\gamma$ are contained in the cone and the Lipschitz constant of $\gamma$ is uniform bounded by the length of the cone.
We denote by $\ell(\gamma)$ the length of $\gamma$.

\begin{thm}\label{thm-u-arc}
For every u-arc $\gamma$, the length of $f^n(\gamma)$ grows exponentially.
\end{thm}

The proof  of the theorem above is an adaptation of \cite{Pujals-Sambarino} in our context. For completeness we present here the details. 

\begin{defi}
We say that an u-arc $\gamma$ is a $\delta$-$u$-arc provided the next condition holds:
\begin{align}
\ell(f^n(\gamma))\leq \delta,\ \  \text{for every} \ \ n\geq 0.
\end{align}
\end{defi}
In other words, a $\delta$-$u$-arc is an {\it{u-arc}} such  that the length of the forward iterates remain bounded.

\medskip

The main idea for proving Theorem \ref{thm-u-arc} is to ensure that there is no $\delta$-$u$-arc. Hence, we will study some consequences in case this kind of arcs exist. The first one is the following.

\begin{lemma}\label{lema-PS}
There exist $ 0<\lambda<1,\, \delta>0, \, C>0$, and $n_0\geq 1$ such that given $\delta$-$u$-arc $\gamma$ and $x \in f^{n_0}(\gamma)$ one has that:
\begin{align}\label{stable-direction}
\|Df^j\mid_{E(x)}\|<C\lambda^j, \ \
\text{for every} \ \ j\geq 1.
\end{align}
\end{lemma}

\begin{proof}
By the domination property, there exists $\ell \geq 1$ such that
\begin{align}\label{cone-u-arc}
\|Df^{\ell}\mid_{E(x)}\|\leq \frac{1}{2}\|Df^{\ell}(v)\|,
\end{align}
for all $v \in \mathscr{C}(x)$ with $ \|v\|=1,$ and  $x \in M$. By continuity, given a small $a>0$ there are $\delta_0, \theta_1 >0$ such that for every $x, y$ with
$d(x,y)<2\delta_0$ and $v \in \mathscr{C}(x), w \in \mathscr{C}(y)$, with $\varangle(w,v)<\theta_1$ \footnote{The angle between $v$ and $w$ is calculated using the local identification $TM\mid_U=U\times \R^2$.}, follows that:
\begin{itemize}
\item[-] $\|Df_x(v)\|\geq (1-a)\|Df_y(w)\|$; 
\item[-] $\|Df\mid_{E(y)}\|,\|Df\mid_{E(x)}\|< a$, if $x,y \in B(\mathrm{Cr}(f),\delta_0)$; and,
\item[-] $\|Df\mid_{E(x)}\|\leq (1+a) \|Df\mid_{E(y)}\|$, if $y \notin B(\mathrm{Cr}(f),\delta_0)$,
\end{itemize}
where $B(\mathrm{Cr}(f),\delta_0)=\{x \in M: d(x,\mathrm{Cr}(f))<\delta_0\}$, recalling that $\mathrm{Cr}(f)$ is the set of the critical points.

\medskip

Since $\mathscr{C}$ is a continuous $Df$-invariant  cone-field, we fix $0<\delta<\delta_0$ and $n_0\geq 1$ large enough so that for every $x,y \in M$ with $d(x,y)\leq \delta$ and $v \in f^{n_0}(\mathscr{C}(x)),\, w \in f^{n_0}(\mathscr{C}(y))$, we have that $\varangle(w,v)<\theta_1$. Thus, taking $\beta>0$ so that $1<(1-a)(1+\beta)<2$, one can obtain for every $\delta$-$u$-arc $\gamma$ and $t \in [0,1]$ that:
\begin{align}\label{u-arc}
\|Df^k\mid_{\R\langle \gamma'_{n_0}(t)\rangle}\|\leq (1+\beta)^k
\end{align}
for $k$ sufficiently large and $\gamma_n(t)=f^n(\gamma(t))$.

Indeed, assume that $\gamma:=\gamma_{n_0}$ is parametrized by the arc  length. Suppose, by contradiction, that there exists a
sequence $(k_j)_{j}$ going to infinity as $j$ goes to infinity such that:
\begin{align*}
\|Df^{k_j}\mid_{\R\langle \gamma'(t_j)\rangle}\|>(1+\beta)^{k_j}.
\end{align*}

Since $\ell(\gamma_n)\leq \delta$, one has that $d(\gamma_n(t),\gamma_n(s))\leq \delta$ for every $t,s\in [0,1]$. Thus, for every $t\in [0,1]$, 
\begin{align*}
\|Df^{k_j}(\gamma'(t))\|(1-a)^{k_j}\|Df^{k_j}(\gamma'(t_{j}))\|\geq ((1-a)(1+\beta))^{k_j}.
\end{align*}
In particular, we would have that $\ell(f^{k_j}(\gamma))\geq ((1-a)(1+\beta))^{k_j}\ell(\gamma)$, contradicting that  $\gamma$ is a $\delta$-$u$-arc. Hence, the equation \eqref{u-arc} holds.

Finally, choosing $\beta > 0$ such that $1<((1-a)(1+\beta))^{\ell}<2$ follows that for $k\geq 1$ large enough,
\begin{align}
\|Df^{k\ell}\mid_{E(x)}\|\leq \lambda^{k\ell}, \ \ \text{for all} \ \ x \in \gamma,
\end{align}
where $\lambda \in (0,1)$ is chosen so that $((1-a)(1+\beta))^{\ell}/2<\lambda^{\ell} <1$. 
This shows that equation \eqref{stable-direction}
holds.
\end{proof}

Note that $\lambda \in (0,1)$ above can be taken uniformly for every $\delta$-$u$-arc with $0<\delta<\delta_0$. 
From now on, we fix $0<\lambda<1$ and $0<\delta<\delta_0$ as in Lemma \ref{lema-PS} and $\lambda'>0$ such that $(1+a)\lambda<\lambda'<1$. Up to replacing by iterates, we assume for the next result that $\gamma$ satisfies  Lemma \ref{lema-PS}.
The following result provides the existence of the local stable manifold for each point belonging to a $\delta$-$u$-arc $\gamma$.

\begin{lemma}\label{lemma-stable-mfld}
There exists $\alpha>0$ such that for every $x \in \gamma$, there is a unique curve $\xi_x:(-\alpha,\alpha)\to M$ orientation  preserving satisfying:
\begin{equation}\label{eq-solution}
\left\{\begin{array}{ll}
\xi_x'(t) \in E(\xi(t)) \ \ \text{with} \ \ \|\xi_x'(t)\|=1;\\
\xi_x(0)=x.
\end{array}\right.
\end{equation}
\end{lemma}


\begin{proof}
Peano's Theorem guarantee the existence of an interval $I=(-\alpha_0,\alpha_0)$ where there exists at least one solution defined on it. Suppose by contradiction that $\xi_0,\xi_1:I \to M$ are two different solutions of \eqref{eq-solution}. For $0<\alpha<\alpha_0$, consider 
the set $D_{\alpha}$
 of all points $y \in M$ such that there exists a solution $\xi_x$ of \eqref{eq-solution} with $\xi_x(\alpha)=y$. It is easy to see that 
$D_{\alpha}$
is  a connected compact set in $M$. 
Then, by the proof of Lemma \ref{lema-PS},
\begin{align*}
&\|Df\mid_{E(y)}\|,\|Df\mid_{E(x)}\|< a,\, \text{if} \ \ x,y \in B(\mathrm{Cr}(f),\delta_0); \ \ \text{and},\\
&\|Df\mid_{E(x)}\|\leq (1+a) \|Df\mid_{E(y)}\|,\, \text{if}  \ y \notin B(\mathrm{Cr}(f),\delta_0).
\end{align*}

Denote by $\mathcal{Q}$ the region bounded by $\xi_0,\xi_1$ and 
$D_{\alpha}$ 
 as in  Figure \ref{curve-cone}.  
\begin{figure}[!h]
\centering
\includegraphics[scale=0.9]{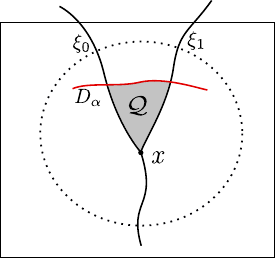}
\caption{Bounded region $\mathcal{Q}$. }\label{curve-cone}
\end{figure}
Since for all $y \in \mathcal{Q}$ there exists  a solution $\xi_x$ of \eqref{eq-solution} with $y=\xi_x(t)$ for some $0<t<\alpha$, one has that
\begin{align}\label{diam-Q}
\begin{split}
d(f^n(y),f^n(x))\leq \ell(f^n\circ\xi_x)\mid_{0}^t &\leq \|Df^n\mid_{E(x)}\|(1+a)^n\ell(\xi_x)\\
&\leq C\lambda^n(1+a)^n \alpha \leq C \lambda'^n\alpha.
\end{split}
\end{align}
In other words, the diameter of the set $f^{n}(\mathcal{Q})$ goes to zero as $n$ goes to infinity. Thus, by transitivity of $f$, we can choose $n\geq 1$ such that the diameter of $f^n({\rm{int}(\mathcal{Q})})$ is small enough so that  $f^n({\rm{int}(\mathcal{Q})})\subseteq \mathcal{Q}$ which contradicts the transitivity.
\end{proof}


Finally, for $\varepsilon >0$ small enough, denote  the set $\{\xi_x(t): t \in (-\varepsilon,\varepsilon)\}$ by $W^{s}_{\varepsilon}(x)$.
In particular,
\begin{align*}
y \in W^{s}_{\varepsilon}(x) \ \ \Longrightarrow \ \ 
d(f^n(x),f^n(y)) \to 0, \,\, \text{as} \,\, n \to +\infty.
\end{align*}

\begin{rmk}
It should be noted that $\alpha_0>0$ can be taken uniform in Peano's Theorem. In particular, fixed $\delta_0>0$ and $0<\alpha<\alpha_0$, we have that for $\delta$-$u$-arc $\gamma$ with $0<\delta<\delta_0$, the size of the local stable manifold can be taken uniform.
By uniqueness of the solution of the equation \eqref{eq-solution} and the invariance of $E$ by $Df$, we have that for $n\geq 1$ large enough,
\begin{align}\label{eq-stable}
f^n(W^s_{\varepsilon}(x))\subseteq W^s_{\varepsilon}(f^n(x)), \forall x \in \gamma.
\end{align}
\end{rmk}

Let us call by {\it box} the following open set,
\begin{align}
W^{s}_{\varepsilon}(\gamma)=\bigcup_{x \in \gamma} W^s_{\varepsilon}(x).
\end{align}

The next result characterizes the dynamic of a $\delta$-$u$-arc and ensures that its existence  is an obstruction for transitivity. The proof is  inspired in \cite[Theorem 3.1]{Pujals-Sambarino} and gives a characterization of the $\omega$-limit of a $\delta$-$u$-arc $\gamma$, denoted by $\omega(\gamma)$. For completeness we present the  proof adapted to our setting as follows.

\begin{thm}\label{PS}

If $\gamma$ is a $\delta$-$u$-arc with $0<\delta \leq \delta_0$, then one of the following properties holds:
\begin{enumerate}[label=$\arabic*.$]
\item $\omega(\gamma) \subseteq \tilde{\beta}$, where $\tilde{\beta}$ is a periodic
simple closed curve normally attracting.\\

\item There exists a normally attracting periodic arc $\tilde{\beta}$ such that $\gamma \subseteq W^s_{\varepsilon}(\tilde{\beta})$.\\

\item $\omega(\gamma) \subseteq {\rm{Per}}(f)$, where ${\rm{Per}}(f)$ is the set of the periodic points of f. Moreover,
one of the periodic points is either a semi-attracting periodic point or an attracting one (i.e.,
the set of  points $y \in M$ such that $d(f^n(p),f^n(y)) \to 0$ contains an open set in $M$).
\end{enumerate}
\end{thm}

\begin{proof}
Let $\gamma_n:=f^n(\gamma)$. By transitivity of $f$, there exists $n_0\geq 1$ large enough verifying the equation \eqref{eq-stable} and
\begin{align}\label{intesection}
W^{s}_{\varepsilon}(\gamma)\cap W^s_{\varepsilon}(\gamma_{n_0})\not=\emptyset.
\end{align}
Consequently, $W^{s}_{\varepsilon}(\gamma_{_{(k-1)n_0}})\cap W^s_{\varepsilon}(\gamma_{_{kn_0}})\not=\emptyset$.

\medskip

If $\ell(\gamma_{_{kn_0}})$ goes to zero as $k$ goes to infinity, then $\omega(\gamma)$ consist of a periodic orbit.

\medskip

Indeed, if $\ell(\gamma_{kn_0}) \to 0$,  then $\ell(\gamma_n) \to 0$ as $n \to \infty$. Let $p$ be an accumulation point of $\gamma_{kn_0}$. That is, there exist a subsequence $(k_j)_j$
and $x \in \gamma$ such that $f^{k_jn_0}(x)\to p$. In particular, as $\ell(\gamma_n)\to 0$, one has
$\gamma_{_{k_jn_0}} \to p$ as $j \to \infty$, and by
\eqref{intesection}, it follows that the limit is independent of the subsequence $(k_j)_j$, and so, we have $\gamma_{_{kn_0}} \to p$ as $k \to \infty$. Hence, $\gamma_{_{kn_0+r}} \to f^{r}(p)$ for $0\leq r \leq n_0-1$, implying that $p$ is a periodic point. Thus, for every $x \in \gamma$ we have that $\omega(x)$ consist only of the periodic orbit of $p$. This proves item ($3$).

\medskip

On the other hand, if $\lim_{k\to +\infty}\ell(\gamma_{_{kn_0}})\geq c> 0$, then there exists a subsequence $(k_j)_j$ such that $\gamma_{_{k_j n_0}} \to \beta$, where $\beta$ is an arc which is at least $C^1$ and
tangent to $\mathscr{C}^{\ast}_E$, since
\begin{equation*}
\gamma'(t^-)=\dis\lim_{s \to 0 (s<0)}\frac{\gamma(t+s)-\gamma(t)}{s} \ \ \text{and} \ \
\gamma'(t^+)=\dis\lim_{s\to 0 (s>0)}\frac{\gamma(t+s)-\gamma(t)}{s}
\end{equation*}
belong to $\mathscr{C}^{\ast}_E$, and so
$\dis\lim_{j\to \infty}Df^{k_jn_0}(\gamma'(t^-)) =\dis\lim_{j\to \infty}Df^{k_jn_0}(\gamma'(t^+))$.

Observe that $f^{n_0}(\beta)$ is the limit of $f^{k_jn_0}(\gamma_{n_0})$ and, by \eqref{intesection}, $\beta \cup f^{n_0}(\beta)$ is a $C^1$-curve.
Let
\begin{align}\label{beta}
\tilde{\beta}=\bigcup_{k\geq 0}f^{kn_0}(\beta).
\end{align}

Let us prove that there are two possibilities:
either $\tilde{\beta}$ is an arc or a simple closed curve. 

\medskip

In order to prove it, first note that for every $k\geq 0$ the curve $f^{kn_0}(\beta)$ is a $\delta$-$u$-arc. In particular, for each $x \in \tilde{\beta}$
there exists $\varepsilon(x)>0$ such that $W^s_{\varepsilon(x)}(x)$ is the local stable manifold for $x$. Thus, the set
\begin{align*}
W^s(\tilde{\beta})=\bigcup_{x \in \tilde{\beta}} W^s_{\varepsilon(x)}(x)
\end{align*}
is a neighborhood of $\tilde{\beta}$.

Finally, we show that given $x \in \tilde{\beta}$ there exists a neighborhood $B(x)$ of $x$ in $M$ such that $B(x)\cap \tilde{\beta}$ is an arc which implies that $\tilde{\beta}$ is a simple closed curve or an interval.

Let $x \in \tilde{\beta}.$ Note that $x \in f^{k_1n_0}(\beta)$ for some $k_1\geq 0$ by \eqref{beta}. Taking $I$ an open interval in $f^{k_1n_0}(\beta)$ containing $x$ and  $B(x)$ neighborhood of $x$ with $\mathrm{diam}(B(x))<c/2$ such that
\begin{align*}
B(x) \subseteq W^s(\tilde{\beta}) \ \ \text{and} \ \ B(x)\cap \beta_1 \subseteq I,
\end{align*}
where $\beta_1$ is an interval containing $f^{k_1n_0}(\beta)$ with $\ell(\beta_1)\geq c/2$.

\medskip

Now, we claim that for every $y \in \tilde{\beta}\cap B(x)$, one has that $y \in I$.

\medskip

Indeed, assume, without loss of generality, that $y \in f^{k_2n_0}(\beta)$. Since
\begin{align*}
f^{k_{\ast}n_0}(\beta)=\dis\lim_{j\to \infty} f^{k_jn_0+k_{\ast}n_0}(\gamma) \ \ \text{for} \ \ \ast=1,2,
\end{align*}
and as both have nonempty intersection with $B(x)$, we conclude that, for some $j$, $f^{k_jn_0+k_1n_0}(\gamma)$ and $f^{k_jn_0+k_2n_0}(\gamma)$ are linked by a local stable manifold. Hence $f^{k_1n_0}(\beta)\cap f^{k_2n_0}(\beta)$ is an arc  $\beta'$ tangent to $\mathscr{C}^{\ast}_E$. Therefore,  $y \in B(x)\cap \beta' \subseteq \beta_1$. This completes the proof that $\tilde{\beta}$ is an arc or a simple closed curve. Furthermore, since $f^{n_0}(\tilde{\beta})\subseteq \tilde{\beta}$, it follows that for every $x \in \gamma,\, \omega(x)$ is the $\omega$-limit of a point in $\tilde{\beta}$, hence item (1) or item (2) holds, completing the proof of the theorem.
\end{proof}

As a consequence we have the following result.

\begin{cor}\label{corPS}
There is not $\delta$-$u$-arc provided $\delta$ small.
\end{cor}

\begin{proof}
From Theorem \ref{PS} follows that the $\omega$-limit of a $\delta$-$u$-arc  is either a periodic simple closed curve normally attracting, or a semi-attracting periodic point, or
there exists a normally attracting periodic arc. In any case, it contradicts the fact that $f$ is transitive.
\end{proof}

\begin{lemma}\label{growing-curve}
For  $\delta>0$ small enough, there exists $n_0 \geq 1$ such that for every $u$-arc $\gamma$ with $\delta/2 \leq \ell(\gamma) \leq \delta$, one has that  $\ell(f^{n}(\gamma))\geq 2\delta$ for some $0\leq n \leq n_0$.
\end{lemma}

\begin{proof}
Otherwise, we should have a sequence $(\gamma_n)_n$ of $u$-arc so that for each $n\geq 1$, one has that: $$\ell(f^j(\gamma_n))\leq 2\delta \ \ \text{for every} \ \ 1\leq j \leq n.$$

Since $\gamma_n' \subseteq \mathscr{C}$, we have that the Lipschitz constant of $\gamma_n$ is uniformly bounded. In particular, the family $\{\gamma_n\}_n$ is uniformly bounded and equicontinuous. That is,
\begin{itemize}
\item[-]  $d(\gamma_n(t),\gamma_n(0))\leq \delta$ for every $t \in [0,1]$ and $n \geq 1$;
\item[-] $\forall \varepsilon>0, \exists \nu>0$ such that for every  $n\geq 1$,
$$\forall \,t,s \in [0,1],|t-s|<\nu \Longrightarrow d(\gamma_n(t),\gamma_n(s))<\varepsilon.$$
\end{itemize}
Then, by Arzel{\`a}-Ascoli's Theorem, up to taking a subsequence, $\gamma_n$ converges uniformly to
the  $2\delta$-$u$-arc $\gamma$, since $\gamma$ is a Lipschitz curve with $\ell(f^k(\gamma))\leq 2\delta$ and $\gamma' \subseteq \mathscr{C}$, contradicting the Corollary \ref{corPS}.
\end{proof}

Finally, we prove the main result of this section. 

\begin{proof}[Proof of Theorem $\ref{thm-u-arc}$]
Fix $\delta>0$ small enough and $n_0\geq 1$ as in Lemma \ref{growing-curve}.
Note that there exists  $\rho > 0$ such that every $u$-arc $\gamma$ with $\ell(\gamma)$ larger than 
$\delta/2$ verifies that $\ell(f^j(\gamma))\geq \rho$, for every $1\leq j\leq n_0$. Otherwise, there exists a sequence of $u$-arc $\gamma_n$ satisfying $\ell(\gamma_n)\geq \delta/2$ and $\ell(f^{j}(\gamma_n))\!\to\! 0$ as $n \to +\infty$, for some $1\leq j \leq n$. Then, by Arzel{\`a}-Ascoli's Theorem, up to take a subsequence, there exists an $u$-arc $\gamma$ satisfying $\gamma=\lim \gamma_n$ with $\ell(\gamma)\geq \delta/2$ and $\ell(f^j(\gamma))=0$. Therefore, there exists $t \in (0,1)$ such that $\gamma'(t)\in \ker(Df^j)$ which contradicts the fact that $\mathscr{C}$ is transversal to the kernel.

\medskip

Now let us prove that if $\ell(\gamma)\geq \delta$ then $\ell(f^j(\gamma))$ grows exponentially.

\smallskip

By the observation above and Lemma \ref{growing-curve}, there exists $1\leq j_1 \leq n_0$ such that 
$\ell(f^{j_1}(\gamma))\geq 2\delta$. Thus, $f^{j_1}(\gamma)$ can be divided in two $u$-arcs, $\gamma_1$ and $\gamma_2,$ each one with length larger than $\delta$.
Repeating the process, we get that
$\ell(f^{jn_0}(\gamma))\geq 2^j\rho,$ for every $ j\geq 1$, finishing the proof.
\end{proof}


\subsection{Proof of Theorem \ref{thm-C}}
Let us prove the main goal of Section \cref{homotopy}.
Consider the lift of the structure we have on $M$ to its universal cover $\mathbb{R}^2$. Denote by $\widetilde{E}$ and $\widetilde{\mathscr{C}}$, the lifts of the subbundle $E$ and the cone-field $\mathscr{C},$ respectively.
Assume $\widetilde{E}$ orientable. 
 Let $\tilde{f}:\mathbb{R}^2 \to \mathbb{R}^2$ be the lift of $f$. We use tilde to denote the tangent curves  to $\widetilde{\mathscr{C}}$ and points in $\mathbb{R}^2$ as well. 
Consider the following balls centered at a curve $\tilde{\gamma}$ and $\tilde{x}$, respectively,
$$B(\tilde{\gamma},\varepsilon)=\{\tilde{y} \in \R^2: d(\tilde{y},\tilde{\gamma})<\varepsilon\}\ \ \text{and} \ \ B(\tilde{x},\varepsilon)=\{\tilde{y} \in \R^2: d(\tilde{x},\tilde{y})<\varepsilon\}.$$

\begin{lemma}\label{area-curve}
There exist $\varepsilon>0$ and $C>0$ such that for every $\tilde{\gamma}:[0,1] \to \R^2$ of $C^1$ class  with $\tilde{\gamma}'\subseteq \mathscr{C}^{\ast}_{\widetilde{E}}$,
holds
\begin{align}\label{eq-volume}
area(B(\tilde{\gamma},\varepsilon))\geq C \ell(\tilde{\gamma}).
\end{align}
\end{lemma}

\begin{proof}
First, if  $\tilde{\gamma}$  is such that  $\tilde{\gamma}'\subseteq \mathscr{C}^{\ast}_{\widetilde{E}}$, then 
$\tilde{\gamma}$ is a simple curve. That is,  $\tilde{\gamma}:[0,1] \to \R^2$ is injective.
In fact, suppose, without loss of generality, that $\tilde{\gamma}(0)=\tilde{\gamma}(1)$. Let $D$ be a disk such that its boundary $\partial D$ is the curve $\tilde{\gamma}$. Since $\widetilde{E}$ is orientable and transverse to $\partial D$, we may define a non-vanishing vector field on $D$. However, by Poincar{\'e}-Bendixson Theorem, every continuous vector field on $D$ transversal to $\partial D$ has a singularity. Therefore, we cannot have $\tilde{\gamma}(t)=\tilde{\gamma}(s)$ with $t\neq s$ in $[0,1]$.

\smallskip

Second, let us prove that there exists $\varepsilon >0$ so that for any ball $B(\tilde{x},\varepsilon)$ the intersection $B(\tilde{x},\varepsilon)\cap \tilde{\gamma}$ has at most one connected component.

Fix $\varepsilon >0$ small enough such that the tangent curve  to $\widetilde{E}$ passing through the
point $\tilde{x}$ divides $B(\tilde{x},\varepsilon)$ in two connected components. It is possible, because
$\widetilde{E}$ induces a continuous vector field on $\R^2$ and it is bounded. Now, suppose that
$\tilde{\gamma}(t_1) \in B(\tilde{\gamma}(t_0),\varepsilon)$ for some $0 \leq t_0 < t_1\leq 1$.
Since $\tilde{\gamma}' \subseteq \mathscr{C}^{\ast}_{\widetilde{E}}$, we can take a disk $D$ such that the
distribution $\widetilde{E}$ induce a continuous vector field on $D$ ($D$ is a disk whose  boundary
is the union of a tangent curve to $\widetilde{E}$ from $\tilde{\gamma}(t_1)$ to $\tilde{\gamma}(t_0)$ and  $\tilde{\gamma}$).
\begin{figure}[!h]
\centering
\includegraphics[scale=0.6]{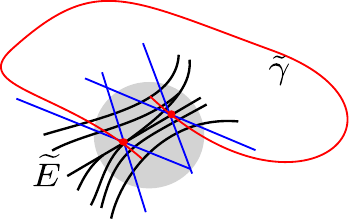}
\caption{The distribution $\widetilde{E}$ and disk $D$.}\label{distribuicao}
\end{figure}
Then, repeating the same arguments one gets, by Poincar{\'e}-Bendixson Theorem, that such vector field has a singularity on $D,$
which is a contradiction, and hence, follows the assertion. Therefore,  we conclude that there exists $\varepsilon>0$ such that $B(\tilde{x},\varepsilon)\cap \tilde{\gamma}$ has at least one component.

\smallskip

Finally, given $L_0\geq 1$ large, up to changing $\varepsilon$, we can assume that every $C^1$ tangent curve to the cone-field with length larger than $L_0$ is not contained in a ball of radius $\varepsilon$. Thus, assume $\ell(\tilde{\gamma})\gg L_0$. Then, consider $k\geq 1$ the largest integer less or equal to $\ell(\tilde{\gamma})/L_0$ and the set
$\{\tilde{x}_1,...,\tilde{x}_k\}$ contained in $\tilde{\gamma}$ such that the curve $\tilde{\gamma}_j$ in $\tilde{\gamma}$ that passes through $\tilde{x}_{j}$ has length $L_0$ and $\{B(\tilde{x}_j,\varepsilon/2)\}_j$ are two-by-two disjoints. Thus, we have that:
\begin{align*}
area(B(\tilde{\gamma},\varepsilon))\geq \dis\sum_{1 \leq j\leq k} area(B(\tilde{x}_j,\varepsilon/2))
\geq C_0 \frac{\ell(\tilde{\gamma})}{2L_0},
\end{align*}
where $C_0$ is the area of the ball of radio $\varepsilon/2$. Therefore, taking $C=C_0/2L_0$ follows the equation \eqref{eq-volume}.
\end{proof}


Now we are able to prove Theorem \ref{thm-C} which is inspired in \cite[Theorem 1.1]{BBI04}.

\begin{proof}[Proof of Theorem \rm{\ref{thm-C}}]
It is well known that there exists a unique square matrix $A$ with  integers entries such that $\tilde{f}=A+\phi$, where $\phi$ is  $\pi_1(M)$-periodic map, that is, $\phi(\tilde{x}+v)=\phi(\tilde{x})$
for every $v \in \pi_1(M)$ and $\tilde{x} \in \R^2$. Assume by contradiction that the absolute value of all the eigenvalues of $A$ are less or equal to one. Thus, the diameter of the images of every compact set under the iterates of $\tilde{f}$ grows sub-exponentially.

Let $B_n$ be a ball centered at $\tilde{x}_n \in \tilde{\gamma}_n$ with radius 
equal to the length of $\tilde{\gamma}_n$ plus $\varepsilon$, where $\tilde{\gamma}_n$ is the image by $\tilde{f}^n$ of a $C^1$-curve $\tilde{\gamma}$ with
$\tilde{\gamma}'\subseteq \mathscr{C}^{\ast}_{\widetilde{E}}$, and $\varepsilon>0$ is given by Lemma \ref{area-curve}.
Since the length
of $\tilde{\gamma}_n$ grows sub-exponentially, we have that the area of $B_n$ grows sub-exponentially. Since $B_n$ contains the neighborhood $B(\tilde{\gamma}_n,\varepsilon)$ of $\tilde{\gamma}_n$ and  Lemma \ref{area-curve}, we have that:
$$area(B_n)\geq area(B(\tilde{\gamma}_n,\varepsilon))\geq C\ell(\tilde{\gamma}_n).$$
However, Theorem \ref{thm-u-arc} implies that the length of $\tilde{\gamma}_n$ grows exponentially, getting a contradiction. Thus, $A$ must have at least one eigenvalue with modulus larger than one.
\end{proof}

Finally, we prove Corollaries \ref{cor1} and \ref{cor2} stated in  Section \cref{intro}.

\begin{proof}[Proof of Corollary \rm{\ref{cor1}}]
The proof follows immediately from the fact that partially hyperbolic endomorphisms admits an unstable cone-field which implies that the iterates of any $C^1$ tangent arc to the unstable cone-field grows exponentially. And so, we can repeat the same argument as in the proof of Theorem \ref{thm-C}.
\end{proof}

\begin{proof}[Proof of Corollary \rm{\ref{cor2}}]
Let $f \in \mathrm{End}^1(M)$ be a robustly transitive endomorphism. 
Note that $f$ either admits a dominated splitting or not.

Assuming $f$ admits a dominated splitting, then Theorem \ref{thm-C} implies that $f$ is homotopic to a linear map having at least one eigenvalue with modulus larger than one, proving our assertion.

On the other hand, if $f$ does not admit a dominated splitting,  Theorem \ref{thm-DS} implies that the set of its critical points is empty, then $f$ is a local diffeomorphism. Observe that if $f$ is approximated by an endomorphism which admits a dominated splitting, then  Theorem \ref{thm-C} implies that $f$ is homotopic to a linear map having at least one eigenvalue with modulus larger than one. Hence, it can be assumed that $f$ is a robustly transitive endomorphism (local diffeomorphism) which has no dominated splitting in a robust way, that is, there exists a neighborhood $\mathcal{W}$ of $f$ in $\mathrm{End}^1(M)$ such that for every $g \in \mathcal{W}$ does not admit a dominated splitting. Since \cite[Theorem 4.3]{LP},  $f$ is volume expanding. Therefore, using the same arguments as in the proof of Theorem \ref{thm-C}, we have that if the absolute value of all the eigenvalues of $A,$ as in Theorem \ref{thm-C}, are less or equal to one, then the area of the ball grows sub-exponentially, contradicting that $f$ is volume expanding, and finishing the proof.
\end{proof}


\section{Expanding direction}\label{sec:exp direction}




Let $f \in \mathrm{End}^1(M)$ be a robustly transitive endomorphism displaying critical points. By Theorem \ref{thm-DS}, $f$ admits a dominated splitting. That is, there exist $\ell \geq 1$ and a splitting $E\oplus F$ of $TM$ so that for all $(x_i)_i \in M_f$ and $i \in \mathbb{Z}$,
\begin{itemize}
\item[-] $Df(E(x_i))\subseteq E(f(x_i)) \ \ \text{and} \ \  Df(F(x_i))=F(f(x_i));$
\item[-] the angle between $E$ and $F$ is uniform bounded away from zero; and,
\item[-] $\|Df^{\ell}\mid_{E(x_i)}\|\leq \frac{1}{2}\|Df^{\ell}\mid_{F(x_i)}\|.$
\end{itemize} 
Or equivalently, there exist a $Df$-invariant continuous subbundle $E$ of $TM$ and a $Df$-invariant continuous cone-field $\mathscr{C}: x \in M \mapsto \mathscr{C}(x,\eta)$ transverses to $E$. Recalling that Proposition \ref{cone-criterion} shows that  both notion of dominated splitting,  ([PH1]) in Section~\ref{intro} and Definition \ref{def-DS} in Section~\ref{section-ph}, are equivalent.


This section is devoted for finishing the proof of Theorem \ref{thm-A}. 
Recalling  Theorem \ref{thm-u-arc}, we have proved so far that the iterates of any arc $\gamma$ tangent to $\mathscr{C}$ grows exponentially.
Thus, it remains to be proved the existence of a real number $\lambda>1$ such that $\|Df^{\ell}\mid_{F(x_i)}\|\geq \lambda$ for all $(x_i)_i \in M_f$ and $i \in \mathbb{Z}$, that is, to show that $F$ is an uniform expanding subbundle.


In order to prove the previous assertion, we assume by contradiction that $F$ is not expanding and use the domination property to prove for every $u$-arc $\gamma$ holds that the subbundle $E$ on $\gamma$ is contracting, recall that $\gamma$ is a tangent arc to $\mathscr{C}(x,\eta)$. Then, taking a small box $W(\gamma)$, we have that its iterates expands on the cone direction and contracts on the $E$ direction. 
Finally, we use that the iterates of the box intersects itself infinitely many times to create, up to a perturbed, a sink, contradicting that the map is  robustly transitive. 

\smallskip

Let us denote by $\gamma_x:[-1,1] \to M$  an $u$-arc of  $C^1$ class with $\gamma_x(0)=x \in M$ having the same order of $[-1,1]$. Follow directly from  the fact that $E$ induce locally a non-vanishing vector field transverse to the cone-field $\mathscr{C}$ and Peano's Theorem that, for all $y \in \gamma_x$, there exists $\xi_y:(-\alpha,\alpha)\to M$ such that $\xi_y(t)\in E(\xi_y(t))$ with $\|\xi_y'(t)\|=1$ for all $t \in (-\alpha,\alpha)$ and $\xi_y(0)=y$. Let $\ubar{\xi},\bar{\xi}:(-\alpha,\alpha) \to M$ be two tangent curves  to the subbundle $E$ with $\ubar{\xi}(0)=\gamma_x(-1)$ and $\bar{\xi}(0)=\gamma_x(1)$.
Define a \textit{$\nu$-box} centered at $\gamma_x$ with $\ubar{\xi}$ and $\bar{\xi}$ as the bottom and top of the box, respectively,  by:
\begin{align}
W_{\nu}(\gamma_x,\ubar{\xi},\bar{\xi})=\left\{\xi_y(t) \in M \left|
\begin{array}{ccc} \ubar{\xi}\leq \xi_y(t) \leq \bar{\xi} \,\, \text{for all} \, \, |t|\leq \nu 
\end{array}\right.\right\},
\end{align}
where $\ell(\xi_y\mid_{[0,\pm\nu]})=\nu$ since $\xi_y$ is parameterized by arc length, and $\ubar{\xi}\leq z \leq \bar{\xi}$ means that every $u$-arc $\gamma_z$ with length larger than $\ell(\gamma_x),$ one has that $\ubar{\xi}(t),\, \bar{\xi}(t') \in \gamma_z$ for some $t, \,t' \in (-\alpha,\alpha)$ and $\ubar{\xi}(t)\leq z \leq \bar{\xi}(t')$ in the induced order by $\gamma_z$. By simplicity, we denote by $\partial^{-}W_{\nu}$ and $\partial^{+}W_{\nu}$, the bottom and the top of the $\nu$-box, respectively, in case there is no confusion about the center, bottom and the top of the $\nu$-box.

\subsection{Existence of a periodic point}
The following result will be used to create a box which is expanding on the cone direction and contracting on the $E$ direction.

\begin{lemma}\label{Existence-pp}
Suppose that there exist $0<\lambda <1,\, C>0$ 
such that for some $x \in M$ holds,
\begin{align}\label{bat-eq} 
\|Df^n\mid_{E(x)}\|\leq C\lambda^n,\, \forall n \geq 1.
\end{align}
Then, for every $\nu>0$ small and $N\geq 1$ large, there is a periodic point $p$ of period $l\geq N$ so that $d(f^j(p),f^j(x))<\nu$ for each $0\leq j \leq l-1$.
\end{lemma}


\begin{proof}
Fix $a>0$ so that $(1+a)\lambda < 1$. By continuity of $y \mapsto \|Df\mid_{E(y)}\|$, there is $\nu_0 > 0$ such that:
\begin{itemize}
\item[-] $\|Df\mid_{E(z)}\|<a$ for every $z \in B(\mathrm{Cr}(f),\nu_0)$;
\item[-] $\|Df\mid_{E(y)}\|\leq (1+a)\|Df\mid_{E(z)}\|, \, \forall y,z \in M$ with $d(y,z)<\nu_0$.
\end{itemize}
Thus,  whenever $d(f^j(y),f^j(x))\leq \nu_0$ for each $0\leq j \leq n-1$, we have that: $$\|Df^n\mid_{E(y)}\|\leq C((1+a)\lambda)^n.$$

Fix $0<\nu<\nu_0$ small enough and $N\geq 1$ large enough so that for every $u$-arc $\gamma$ with $\ell(\gamma)\geq \nu/2$ one has $\ell(f^n(\gamma))\geq 2^n\rho$, for $n\leq N,$ where $\rho > 0$ is given in the proof of Theorem \ref{thm-u-arc}. Moreover, choose $n_0\geq 1 $ so that $2^{n_0}\rho\geq 2\nu$ and $N\geq n_0$. Since the connected components $\gamma^{-}_x=\gamma_x\mid_{[-1,0]}$ and $\gamma^{+}_x=\gamma_x\mid_{[0,1]}$ of an $u$-arc $\gamma_x$ are $u$-arcs as well, we can assume, without loss of generality, that the length of $f^n(\gamma^{\pm}_x)$ are larger than $2\nu_0$ for every $n\geq n_0$. On the other hand, if the $\nu$-box $W_{\nu}(\gamma_x,\ubar{\xi},\bar{\xi})$ is contained in $B(x,\nu_0),$ then for all $\xi_y(t) \in W_{\nu}(\gamma_x,\ubar{\xi},\bar{\xi})$, for all $t\in [0,\nu]$, one verifies that:
\begin{align}\label{eq:xi}
\begin{split}
\ell((f\circ \xi_y)\mid_{[0,t]})& =\int_{0}^{t}\|(f\circ \xi_y)'(s)\|ds \\ & \leq\int_{0}^{t}\|Df\mid_{E(\xi_y(s))}\|\|\xi_y'(s)\|ds\leq C((1+a)\lambda)t, \, \forall t \in [0,\nu].  
\end{split}
\end{align}

Denote by $\gamma_n$ the connected component of $f^n(\gamma_x)$ in the ball $B(f^n(x),\nu_0)$ containing $f^n(x)$, and by $W_n$ the $\nu$-box centered at $\gamma_n$, assuming $W_0=W_{\nu}$. Moreover, suppose that $0<\nu< \nu_0$ small enough so that for every $y \in W_n$ and $u$-arc $\gamma_y$ with $\ell(\gamma_y)\geq 2\nu$ holds $\gamma_y \pitchfork \partial^{\pm}W_n\neq \emptyset$.
We define by induction the following strips: 
\begin{align}
D_0=W_{\nu} \ \ \text{and} \ \ D_n=f(D_{n-1})\cap W_n.
\end{align}
Observe that as $E$ is $Df$-invariant, one has that $f\circ \xi_y$, up to parametrizing by arc length, is an arc of the form $\xi_{f(y)}$. Therefore, for every $y \in D_n$ there exists $\xi_0:[0,t_0]\to B(x,\nu_0)$ such that  $\xi_0(0) \in \gamma_x$ and $\xi_0(t_0)=y_0$ verifying $f^n(y_0)=y$ and $\xi_i(t)=(f^i\circ \xi_0)(t)$ belongs to $D_i$, for all $t \in [0,t_0]$ and $i=1,...,n$. In particular, $\xi_i(t) \in B(f^i(\xi_0(0)),\nu)$ for all $t \in [0,t_0]$, and so, one has that:
\begin{align*}
\ell(\xi_i\mid_{[0,t]}) \leq \int_{0}^{t}\|(f^i\circ \xi_0)'(s)\|ds\leq C((1+a)\lambda)^i t, \, \forall t \in [0,t_0].  
\end{align*}
Thus,  $\mathrm{diam}(D_n)$ goes to zero as $n$ goes to infinity.

Since  $f$ is transitive, we can find $(z_n)_n$ with $z_n$ and $f^{l_n}(z_n)$ converging to $x$ as $n$ goes to infinity. In particular, the $u$-arc $\gamma_{l_n}$ in $D_{l_n}$ verifies  $\gamma_{l_n}\pitchfork \partial^{\pm}D_0\neq \emptyset$.
\begin{figure}[h]
\centering
\includegraphics[scale=0.7]{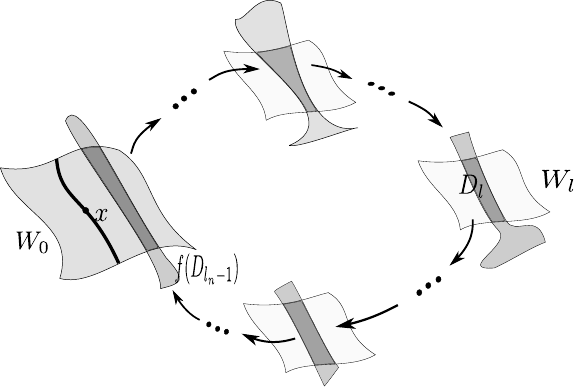}
\caption{The action of $f$.}
\end{figure}

Roughly speaking, for every $l\geq N$ we have that $f$ acts on $D_l$ expanding in the ``vertical" direction (cone-field $\mathscr{C}$) and contracting in the ``horizontal" direction (tangent to $E$).

\smallskip

Therefore, fixing $l_n$ as above, there exists $D_0'\subseteq D_0$ such that $f^{l_n}:D_0' \longrightarrow D_{l_n}$. Then, we can find a periodic point of period $l_n$ as follows:
\begin{itemize}
\item [Step 1.] Repeating the process with $D_0'\cap D_{l_n}$ by $f^{kl_n}$ we obtain a sequence of ``vertical" boxes $(\Gamma_k)_k$ such that:
$$\Gamma_{1}\supseteq \Gamma_{2}\supseteq \cdots \supseteq \Gamma_{k}\supseteq \cdots$$
where $\bigcap_k \Gamma_k=\Gamma$ is a ``vertical"  curve transverse to the box.

\item[Step 2.] Similarly, we have a sequence of ``horizontal" boxes $(\Sigma_k)_k$ such that:
$$\Sigma_1 \supseteq \Sigma_{2} \supseteq \cdots \supseteq \Sigma_k \supseteq \cdots$$
where $\bigcap_k \Sigma_k=\Sigma$ is a ``horizontal" curve.

\item[Step 3.] Therefore, $\Gamma \cap \Sigma=\{p\}$ is a periodic point of period $l_n$.
\end{itemize}

\begin{figure}[h]
\centering
\includegraphics[scale=0.7]{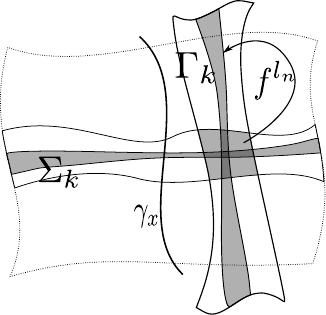}
\caption{Existence of a periodic point.}
\end{figure}

 Observe that $d(x,p)<\nu$ and $f^i(p)\in D_i$ for each $i=0,..., l_{n}-1$. In particular, $d(f^i(p),f^i(x))<\nu$ for $i=0,\dots,l_{n}-1$, finishing the proof.
\end{proof}


Finally, we prove Theorem \ref{thm-A}.

\subsection{Proof of  Theorem A}

In Theorem \ref{thm-DS}, we proved that every $f$ robustly transitive endomorphism displaying critical points 
admits a dominated splitting. 
Let us denote it by $E\oplus F$. Recall that for any $(x_i)_i \in M_f$, the subbundles $E$ and $F$ at $(x_i)_i$ are denoted by $E(x_0)$ and $F(x_0)$, and the action of $f$ on $M_f$ is the \textit{shift map} on $M_f$, that is, $f^n((x_i)_i)=(x_{n+i})_i$ for $n \in \mathbb{Z}$. 

From now on, we will use a classical idea to create sinks using the domination property and vanishing Lyapunov exponent. 

Recalling that we wish to prove that $f$ is partially hyperbolic, suppose instead that $f$ is not a partially hyperbolic endomorphism. Then, Proposition \ref{ph-end} implies the existence of $\tau >0$ such that for every $k \in \mathbb{N}$, there exists $(x_i^k)_i \in M_f$ such that for every $1\leq j \leq k$ holds:
\begin{align*}
\|Df^j\mid_{F(x_0^k)}\|\leq 1+\tau.
\end{align*}
For simplicity, let us denote  $(x_i^k)_i$ by $x^k$. 
Thus, for every $k \in \mathbb{N}$, we define  a measure $\mu_k$ on $M_f$ as follows:
\begin{align*}
\mu_k=\dfrac{1}{k}\sum_{j=0}^{k-1} \delta_{f^j(x^k)},
\end{align*}
where $\delta_{f^j(x^k)}$ denotes the Dirac measure at $(x_{j+i}^k)_i$ in $M_f$. Let $\mu$ be a $f$-invariant measure on $M_f$ which is obtained as an accumulation point of $(\mu_k)_k$. Thus, up to take a subsequence, follows that:
$$\int{\Phi d\mu_k} \to \int{\Phi d\mu}, \ \  \text{for all}\; \Phi \; \text{continuous}.
$$
In particular, $\Phi(\cdot)=\log \|Df\mid_{F(\cdot)}\|$ satisfies:
\begin{align*}
\int{\log \|Df\mid_{F}\|d\mu}&=\lim_{k\to \infty} \int{\log \|Df\mid_{F}\|d\mu_{k}}\\
&=\lim_{k\to \infty} \frac{1}{k}\sum_{j=0}^{k-1} \log \|Df\mid_{F(f^j(x^k))}\|\\
&=\lim_{k\to \infty} \frac{1}{k}\log \|Df^{k}\mid_{F(x_0^{k})}\|\\
&\leq \lim_{k\to \infty} \frac{1}{k} \log (1+\tau) =0.
\end{align*}
On the other hand, using Birkhoff's Ergodic Theorem and Poincar{\'e}'s Recurrence Theorem, there is a recurrent point $(x_i)_i \in M_f$ such that:
\begin{align*}
\lim_{k\to \infty} \frac{1}{k}\sum_{i=0}^{k-1} \log \|Df\mid_{F(x_i)}\|\leq 0.
\end{align*}
Therefore, for every $\varepsilon >0$ there exists $k_0\geq 1$ such that:
\begin{align}\label{eq-lyap-1}
\|Df^k\mid_{F(x_0)}\|=\prod_{i=0}^{k-1}\|Df\mid_{F(x_i)}\|\leq e^{k\varepsilon},\, \text{for all}\;  k \geq k_0.
\end{align}
Since $E\oplus F$ is the dominated splitting for $f$, we have that there exists $C>0$ such that
for every $ (x_i)_i \in M_f$ and $i \in \mathbb{Z}$ hold that:
\begin{align}\label{eq-lyap-2}
\|Df^k\mid_{E(x_i)}\|\leq C\left(\frac{1}{2}\right)^{k}\|Df^k\mid_{F(x_i)}\|,  \, \text{for all}\;  k\geq 1.
\end{align}
In particular, choosing $\varepsilon>0$ small enough so that $\lambda_0=e^{\varepsilon}/2<1$, we get by equations \eqref{eq-lyap-1} and \eqref{eq-lyap-2} that:
\begin{align*}
\|Df^k\mid_{E(x_0)}\|\leq C\lambda_0^k, \, \text{for all}\; k \geq k_0.
\end{align*}
In other words, up to change the constant $C>0$, we have that,
\begin{align}
\|Df^k\mid_{E(x_0)}\|\leq C\lambda_0^k,  \, \text{for all}\; k \geq 1.
\end{align}
Therefore, applying  Lemma \ref{Existence-pp}, there exists a periodic point $p$ of period $k$ large enough so that the eigenvalues of $Df^k$ at $p$  in modulus are at most $e^{k\varepsilon}\approx (1+\tau)$.
On the other hand, if we consider $L_p:\oplus_{i=0}^{k-1}T_{f^i(p)}M \to \oplus_{i=0}^{k-1}T_{f^i(p)}M$ defined by
$$L_p(v_0,v_1,\dots,v_{k-1})=(Df(v_{k-1}),Df(v_0),\dots, Df(v_{k-2})),$$
we have that
$$\omega \ \ \text{is an eigenvalue of} \ \ L_p \Longleftrightarrow \omega^k \ \ \text{is an eigenvalue of} \ \ Df^k_p.$$
Suppose, without loss of generality, that $\omega$ is the eigenvalue with maximum modulus and it satisfies $1-|\omega|^{-1}<\varepsilon$, where $\varepsilon$ is small enough. Then, by \hyperref[Franks-lemma]{Franks' Lemma}, there exists  
a perturbation $h$ of $f$ such that $h^i(p)=f^i(p)$ and $Dh=(|\omega|^{-1}-\varepsilon)Df$ at $f^i(p)$. Hence, $p$ is a sink for $h$ which contradicts that $h$ is transitive, recalling that $f$ is robustly transitive. This proves  Theorem \ref{thm-A}. \qed


\subsection*{Acknowledgments:}
The authors are grateful to E. Pujals, L. Mora, R. Potrie, and S. Luzzatto for
useful and encouraging conversations and suggestions. The authors
are also grateful for the nice environment provided by IMPA, PUC-Rio, UdelaR, UFBA, UFAL and ICTP
during the preparation of this paper. The first author was partially supported with CNPq-IMPA funds and ULA-Venezuela, and the second author by CNPq-IMPA, UFAL and ICTP. 




 \bibliographystyle{alpha}
 \bibliography{Thesis-Biblio}

\begin{thebibliography}{DPU99}

\bibitem[AH94]{AH}
N.~Aoki and K.~Hiraide.
\newblock {\em Topological theory of dynamical systems}.
\newblock North-Holland Mathematical Library, 1st edition, 1994.

\bibitem[BBI04]{BBI04}
pages 307--312. Cambridge Univ. Press, Cambridge, 2004.

\bibitem[BD96]{BD}
C.~Bonatti and L.J. D{\'{\i}}az.
\newblock Persistent nonhyperbolic transitive diffeomorphisms.
\newblock {\em Ann. of Math. (2)}, 143(2):357--396, 1996.

\bibitem[BDP03]{BDP}
C.~Bonatti, L.~J. D{\'{\i}}az, and E.~Pujals.
\newblock A {$C^1$}-generic dichotomy for diffeomorphisms: weak forms of
  hyperbolicity or infinitely many sinks or sources.
\newblock {\em Ann. of Math. (2)}, 158(2):355--418, 2003.

\bibitem[BR13]{Berger-Rovella}
P.~Berger and A.~Rovella.
\newblock On the inverse limit stability of endomorphisms.
\newblock In {\em Annales de l'Institut Henri Poincare (C) Non Linear
  Analysis}, volume~30, pages 463--475. Elsevier, 2013.

\bibitem[CP]{Crovisier-Potrie}
S.~Crovisier and R.~Potrie.
\newblock Introduction to partially hyperbolic dynamics.
\newblock {\em Notes for a minicourse in the School on Dynamical Systems 2015
  at ICTP}.

\bibitem[DPU99]{DPU}
L.~J. D{\'\i}az, E.~Pujals, and R.~Ures.
\newblock Partial hyperbolicity and robust transitivity.
\newblock {\em Acta Mathematica}, 183(1):1--43, 1999.

\bibitem[Fra71]{Franks}
J.~Franks.
\newblock Necessary conditions for stability of diffeomorphisms.
\newblock {\em Trans. Amer. Math. Soc.}, 158:301--308, 1971.

\bibitem[ILP16]{ILP}
J.~Iglesias, C.~Lizana, and A.~Portela.
\newblock Robust transitivity for endomorphisms admitting critical points.
\newblock {\em Proc. Amer. Math. Soc.}, 144(3):1235--1250, 2016.

\bibitem[LP13]{LP}
C.~Lizana and E.~Pujals.
\newblock Robust transitivity for endomorphisms.
\newblock {\em Ergodic Theory and Dynamical Systems}, 33:1082--1114, 8 2013.

\bibitem[LR19]{CW2}
C.~Lizana and W.~Ranter.
\newblock New classes of {$C^1$} robustly transitive maps with persistent
  critical points.
\newblock {\em arXiv:1902.06781}, 2019.

\bibitem[Ma{\~n}82]{Mane-Closinglemma}
R.~Ma{\~n}{\'e}.
\newblock An ergodic closing lemma.
\newblock {\em Annals of Mathematics}, 116(3):503--540, 1982.

\bibitem[Pot12]{Potrie}
R.~Potrie.
\newblock {\em Partial Hyperbolicity and attracting regions in 3-dimensional
  manifolds}.
\newblock PhD thesis, PEDECIBA-Universidad de La Republica-Uruguay, 2012.

\bibitem[PS07]{Pujals-Sambarino}
E.~Pujals and M.~Sambarino.
\newblock Integrability on codimension one dominated splitting.
\newblock {\em Bull. Braz. Math. Soc. (N.S.)}, 38(1):1--19, 2007.

\end{thebibliography}



\end{document}